\documentclass[12pt,a4paper]{article}

\usepackage{amsmath,amsfonts,amsthm,amssymb}
\usepackage{mathrsfs} 
\usepackage{parskip}
\usepackage[T1]{fontenc}
\usepackage[utf8]{inputenc}
\usepackage{authblk}
\usepackage{enumerate}
\usepackage{graphicx}
\usepackage{xcolor}
\usepackage{a4wide}
\allowdisplaybreaks

\newtheorem{theorem}{Theorem}

\newtheorem{remark}{Remark}
\newtheorem{proposition}{Proposition}

\let\ds\displaystyle
\let\scr\mathscr
\let\goth\mathfrak

\def\zT{{\vphantom{\widetilde T} T}}
\def \Liminf{\mathop{\underline{\lim}}\limits}

\def\AA{\mathbb{A}}

\def\Pb{\mathbf{P}}

\def\Ex{\mathbf{E}}
\def\Pb{\mathbf{P}}

\def\KK{\mathbb{K}}

\def\UU{\mathbb{U}}

\def\1{\mbox{1\hspace{-.25em}I}}
\begin{document}
\title{Hidden Ergodic Ornstein-Uhlenbeck Process and Adaptive Filter}
\author{ \textsc{Yury A. Kutoyants}\\ {\small Le Mans University, Le Mans,
    France }\\
{\small     National Research University ``MPEI'', Moscow, Russia }\\ 
{\small Tomsk     State University, Tomsk, Russia }\\
 }

\date{}

\maketitle
\begin{abstract}
 The model of partially observed linear stochastic differential equations
 depending on some unknown parameters is considered. An approximation of the
 unobserved component is proposed. This approximation is realized in three
 steps. First an estimator of the method of moments of unknown parameter is
 constructed. Then this estimator is used for defining the One-step
 MLE-process and finally the last estimator is substituted to the equations of
 Kalman-Bucy (K-B) filter. The solution of obtained K-B equations provide us
 the approximation (adaptive K-B filter).  The asymptotic properties of all
 mentioned estimators and MLE and Bayesian estimators of the unknown
 parameters are described. The asymptotic efficiency of the proposed adaptive
 filter is shown.
\end{abstract}


\section{Introduction}

We are given  a  linear partially observed system
 \begin{align}
\label{1}
{\rm d}X_t&=f\left(\vartheta,t \right)\,Y_t\,{\rm d}t+\sigma\left(t\right)\, {\rm d}W_t,\qquad\qquad 
\; X_0,\quad \qquad t\geq 0,\\
 {\rm d}Y_t&=-a\left(\vartheta,t
\right)\,Y_t\,{\rm d}t+b\left(\vartheta,t \right)\,{\rm d}V_t,\qquad\quad
Y_0,\qquad\quad t\geq 0,
\label{2}
\end{align}
where the Wiener processes $W_t,t\geq 0$ and $V_t,t\geq 0$ are
independent. The observations are $X^T=\left(X_t,0\leq t\leq T\right)$ and the
Ornstein-Uhlenbeck process $Y^T=\left(Y_t,0\leq t\leq T\right)$ is {\it
  hidden}. The initial values $X_0\sim {\cal N}\left(0,d_x^2\right), Y_0\sim
{\cal N}\left(0,d_y^2\right)$ do not depend of Wiener processes. The functions
$f\left(\vartheta,t \right),\sigma \left(t\right),a\left(\vartheta,t
\right),b\left(\vartheta,t \right) $ are supposed to be known and smooth
w.r.t. the unknown parameter $\vartheta \in\Theta\subset {\cal R}^d,d\geq 1 $. 

It is known since 1961 \cite{KB61} that the process $m\left(\vartheta
,t\right)=\Ex_\vartheta \left(Y_t|X_s,0\leq s\leq t\right) $,$0\leq t\leq T$
(with known $\vartheta $)
satisfies the equations of Kalman-Bucy filtering, (see, e.g.  \cite{LS01})
\begin{align}
\label{3}
{\rm d}m\left(\vartheta ,t\right)&=-\left[a\left(\vartheta,t
  \right)+\frac{\gamma \left(\vartheta ,t\right)f\left(\vartheta,t
    \right)^2}{\sigma\left(t\right) ^2} \right] m\left(\vartheta ,t\right){\rm
  d}t+\frac{\gamma \left(\vartheta ,t\right)f\left(\vartheta
  ,t\right)}{\sigma\left(t\right)   ^2} {\rm d}X_t,\\
\label{4}
\frac{\partial \gamma \left(\vartheta ,t\right)}{\partial
  t}&=-2a\left(\vartheta,t \right)\gamma \left(\vartheta ,t\right)-\frac{\gamma
  \left(\vartheta ,t\right)^2f\left(\vartheta ,t\right)^2}{\sigma\left(t\right)
  ^2}+b\left(\vartheta,t \right)^2 ,
\end{align}
with the initial values  $m_0= \Ex_\vartheta\left( Y_0|X_0\right)$ and
$\gamma _0=\Ex_\vartheta \left(Y_0-m\left(\vartheta ,0\right)\right)^2 $.

We are interested by the problem of on-line estimation of the conditional
expectation (random function) $m\left(\vartheta ,t\right), 0<t\leq T$ in the
different situations, where $\vartheta $ is unknown. For example,
$f\left(\vartheta,t \right)=\vartheta $ and $a\left(\vartheta,t \right)=a$,
$b\left(\vartheta,t \right)=b$, where $a$ and $b$ are known.  The usual
behavior in such situations is to estimate first the unknown parameters and
then to substitute these estimators in the equations \eqref{3}-\eqref{4}. Due
to importance of such models in a large diversity of applied problems there
exists a wide literature devoted to adaptive filtering but mainly for discrete
time versions of the model \eqref{1}-\eqref{2}, see, e.g., \cite{BC93},\cite{BRR85},
\cite{Hay14}, \cite{Liung79},\cite{Meh70},\cite{S08} and references
therein. There is a large diversity of the models (linear and non linear) and
the methods of adaptive filtering. The most interesting are of course the
algorithms of on-line adaptive filters. The studied algorithms are mainly
verified with the help of numerical simulations, which show the reasonable
behavior of the adaptive filters.  The presented here continuous time
partially observed systems observed in continuous time were less studied. The
asymptotic properties (consistency, asymptotic normality, convergence of
moments and asymptotic efficiency) of the MLE for the homogeneous model
\begin{align}
\label{5}
{\rm d}X_t&=f\left(\vartheta \right)\,Y_t\,{\rm d}t+\sigma \, {\rm
  d}W_t,\qquad\qquad  \quad 
\; X_0,\quad \qquad 0\leq t\leq T,\\
 {\rm d}Y_t&=-a\left(\vartheta \right)\,Y_t\,{\rm d}t+b\left(\vartheta \right)\,{\rm d}V_t,\qquad\qquad
Y_0,\qquad\quad t\geq 0,\label{6}
\end{align}
 were described in the works \cite{Kut84} (estimation of
 $f\left(\vartheta\right) =f= \vartheta$, i.e., $a\left(\vartheta \right)=a,
 b\left(\vartheta \right)=b$, where $a$ and $b$ are known and $f$ is unknown
 parameter) and \cite{KS91} (estimation of $\vartheta =\left(f,a\right))$. See as
 well \cite{BaB84}, where the consistency of the MLE in the case of infinite
 dimensional linear systems was established. Remark that in \cite{Kut04} the
 properties of the MLE of all one-dimensional parameters for the model
 \eqref{5}-\eqref{6} ($\vartheta =f$, $\vartheta =a$, $\vartheta =b$) were
 described too. Another approach called EM (expectation-maximization) was
 developped in \cite{DZ86}. 

Note as well that the consistent estimation of the parameters $\vartheta
=\left(f,b\right)$ or $\vartheta =\left(f,a,b\right) $ is impossible because
the model \eqref{5}-\eqref{6} depends on the product $fb$.
 
The considerd in this work model is a particular case of the large class of
hidden Markov processes. Some general results concerning asymptotic behavior
of estimators can be found in \cite{BRR98} and \cite{CMR05}.

Our goal is to obtain a good recurrent approximation $m_t^\star,0<t\leq T$ of
the process $m\left(\vartheta ,t\right),0<t\leq T$ in the case of the
homogeneous partially observed system \eqref{5}-\eqref{6} and to discuss the
question of asymptotic efficiency of adaptive filters.

The equations \eqref{3}-\eqref{4} became
\begin{align}
\label{7}
{\rm d}m\left(\vartheta ,t\right)&=-\left[a\left(\vartheta
  \right)+\frac{\gamma \left(\vartheta ,t\right)f\left(\vartheta
    \right)^2}{\sigma ^2} \right] m\left(\vartheta ,t\right){\rm
  d}t+\frac{\gamma \left(\vartheta ,t\right)f\left(\vartheta \right)}{\sigma
  ^2} {\rm d}X_t,\qquad m\left(\vartheta ,0\right)\\
\label{8}
\frac{\partial \gamma \left(\vartheta ,t\right)}{\partial
  t}&=-2a\left(\vartheta \right)\gamma \left(\vartheta ,t\right)-\frac{\gamma
  \left(\vartheta ,t\right)^2f\left(\vartheta \right)^2}{\sigma
  ^2}+b\left(\vartheta \right)^2 ,\qquad \quad \gamma \left(\vartheta ,0\right).
\end{align}

 We have to emphasis,
that this model  can be considered as the simplest example
of models studied in Kalman-Bucy filtering theory. The processes $X_t,t\geq 0$
and $Y_t,t\geq 0$ are usually vectors and the Riccati equation \eqref{4} is
written for matrices \cite{Jaz70}, \cite{LS01}, \cite{Hay14},\cite{S08}. We suppose that
the proposed in this work algorithms can 
be developed for more general models and the only obstacle there can be the
construction of the preliminary estimators. If preliminary consistent
estimators are found then the One-step MLE-processes and equations for
adaptive filters can be written relatively easy.

We need a good estimator of $\vartheta $, which can be easily calculated for
each $t$ and can be used for construction of $m_t^\star,0<t\leq T$. Note
that this can not be the MLE $\hat\vartheta _t, 0<t\leq T$  or Bayesian estimator (BE)
$\tilde\vartheta _t, 0<t\leq T$ defined by the relations 
\begin{align}
\label{MLE}
L(\hat{\vartheta }_t,X^t )=\sup_{\vartheta\in \Theta }L\left(\vartheta,X^t
\right),\qquad \qquad \tilde\vartheta _t=\frac{\int_{\Theta }^{}\vartheta
  p\left(\vartheta \right)L\left(\vartheta,X^t \right){\rm
    d}\vartheta}{\int_{\Theta }^{} p\left(\vartheta
  \right)L\left(\vartheta,X^t \right){\rm d}\vartheta},
\end{align}
where the likelihood ratio function $L\left(\vartheta,X^t
\right),\vartheta \in\Theta  $  is
\begin{align}
\label{LR}
L\left(\vartheta,X^t \right)= \exp\left\{ \int_{0}^{t}\frac{f\left(\vartheta
    \right)m\left(\vartheta,s\right)}{\sigma ^2}\,{\rm d}X_s -
  \int_{0}^{t}\frac{f\left(\vartheta
   \right)^2m\left(\vartheta,s\right)^2}{2\sigma^2}{\rm d}s
  \right\} 
\end{align}
because to solve equations \eqref{3}-\eqref{4} for all (or too many)
$\vartheta $ and to all $t\in (0,T]$ is numerically difficult realizable. Here
  $p\left(\vartheta \right), \vartheta \in\Theta $ is density a priory of the
  random variable $\vartheta $ (Bayesian approach with quadratic loss
  function).

The estimation of   $m\left(\vartheta ,t\right),t\in (0,T]$ in this work
 is realized following the  program: 

{\it
\begin{enumerate}
\item  Calculate a preliminary estimator $\bar\vartheta_\tau  $ on relatively small
  interval of observations $\left[0,\tau \right]$.
\item Using $\bar\vartheta_\tau  $ construct the  One-step MLE-process $\vartheta
  _t^\star, \tau <t\leq T $.
\item As approximation of $m\left(\vartheta ,t\right) $ we propose $m_t^\star
  $ obtained with the help of K-B equations, where $\vartheta $ is replaced by
  $\vartheta _t^\star, \tau <t\leq T $.

\item Estimate the error $m _t^\star-m\left(\vartheta
 ,t\right), \tau <t\leq T  $.
\end{enumerate}
}This means that we have no on-line approximation on the time interval
$\left[0,\tau \right]$, but $\tau /T\rightarrow 0$. Note that the used here
One-step MLE-process is the well-known Le Cam's One-step MLE, in which we 
consider the upper limit of the integral (time $t$) as variable.

We applied this construction  (steps 1-2,  or steps 1-4)
to 6 different models of observations. To Markov processes with discrete time
\cite{KM16} (steps 1-2), to the model of observations \eqref{1}-\eqref{2} with small
noises in the both equations \cite{Kut94}, \cite{KZ21} (1-4). To the model
\eqref{1}-\eqref{2} with small noise in the equation \eqref{1} only
\cite{Kut19a}, \cite{Kut22} (1-4). To the model of hidden telegraph process
\cite{KhK18} (1-2). To the discrete time model of Kalman filtering
\cite{Kut23b} (1-4) and to this model of observations in the one-dimensional
case\cite{Kut19b} (1-2).  In this work we consider the two-dimensional case
(1-2), the adaptive filter and study the error (3-4).

\section{Auxiliary results}

{\bf Solution of Riccati equation}

Let us denote
\begin{align*}
r\left(\vartheta \right)=\left(a\left(\vartheta
\right)^2+\frac{f\left(\vartheta \right)^2b\left(\vartheta \right)^2}{\sigma
  ^2}\right)^{1/2},\qquad\quad \gamma _*\left(\vartheta \right)= \frac{\sigma
  ^2\left[r\left(\vartheta \right)-a\left(\vartheta
    \right)\right]}{f\left(\vartheta \right)^2}.
\end{align*}
  Recall that in the case of constant parameters \eqref{5}-\eqref{6} the
  Riccati equation \eqref{8}  has an explicite solution \cite{A83}
\begin{align}
\label{2-10}
\gamma \left(\vartheta ,t\right)=e^{-2r\left(\vartheta
  \right)t}\left[\frac{1}{\gamma\left(\vartheta ,0\right)-\gamma _*\left(\vartheta
    \right)}+\frac{f\left(\vartheta \right)^2}{2r\left(\vartheta
    \right)\sigma ^2}\left(1-e^{-2r\left(\vartheta
    \right)t}\right)\right]^{-1}+\gamma _*\left(\vartheta \right),
\end{align}
where we suppose that $\gamma\left(\vartheta ,0\right)\not=\gamma _*\left(\vartheta
    \right) $.  According to \eqref{2-10}
 there exists a constant $C>0$ such that 
\begin{align*}
\left|\gamma \left(\vartheta ,t\right)-\gamma _*\left(\vartheta
\right)\right|\leq C\,e^{-2r\left(\vartheta \right)t}\longrightarrow 0,\qquad
      {\rm as}\quad t\longrightarrow \infty .
\end{align*}
This exponential convergence allows us in some
problems without loss of generality  consider the stationary regime supposing $\gamma\left(\vartheta
,0\right)=\gamma _*\left(\vartheta \right) $. This essentially simplifies the
calculations.

 If $m_0\sim {\cal N}\left(0,\gamma _*\left(\vartheta \right)\right)$, i.e.,
 the initial value $\gamma \left(\vartheta ,0\right)=\gamma
 _*\left(\vartheta \right)$, then we have the equalities
\begin{align*}
\frac{\partial \gamma \left(\vartheta ,t\right)}{\partial t}\equiv 0,\qquad     \gamma \left(\vartheta ,t\right)\equiv \gamma _*\left(\vartheta \right)\qquad {\rm
  for}\quad {\rm all}\quad t>0. 
\end{align*}
The value $\gamma _*\left(\theta \right)$ is a positive solution  $z$ of the equation
\begin{align*}
-2a\left(\vartheta \right)z -\frac{z^2f\left(\vartheta \right)^2}{\sigma ^2} +b\left(\vartheta \right)^2=0,
\end{align*}
which is obtained from  the   Riccati equation \eqref{8}:
\begin{align*}
-2a\left(\vartheta \right)\gamma _*\left(\vartheta \right) -\frac{f\left(\vartheta \right)^2\gamma _*\left(\vartheta
  \right)^2}{\sigma ^2} +b\left(\vartheta \right)^2=0.  
\end{align*}
Introduce one more notation $\Gamma \left(\vartheta \right)=\gamma
_*\left(\vartheta \right)\sigma ^{-2}f\left(\vartheta \right)^2 =r\left(\vartheta \right)-a\left(\vartheta \right) $.

In this case  the system \eqref{7}-\eqref{8} is reduced to one equation
\begin{align*}
{\rm d}m\left(\vartheta ,t\right)=-\left[a\left(\vartheta \right)+\gamma \left(\vartheta \right)\right]
  m\left(\vartheta ,t\right){\rm d}t+
\gamma \left(\vartheta \right)f\left(\vartheta \right)^{-1}{\rm d}X_t,\quad m\left(\vartheta ,0\right)=m_0.
\end{align*}

{\bf Stationary representations}

To simplify the
calculations we will use the representations (here $\vartheta _0$ is the true value)
\begin{align*}
m\left(\vartheta_0 ,t\right)&=m_0e^{-a_0\left(\vartheta_0 \right)t}+\frac{\gamma_*
  \left(\vartheta_0\right)f\left(\vartheta_0 \right)}{\sigma }\int_{0}^{t}e^{-a_0\left(\vartheta _0\right)\left(t-s\right)}
{\rm d}\bar W_s,\\ 
m_0&=\frac{\gamma_* \left(\vartheta_0\right)f\left(\vartheta_0\right)}{\sigma
}\int_{-\infty }^{0}e^{a_0\left(\vartheta_0\right)s} {\rm d}\bar W_s,
\end{align*}
where $\bar W_s,s\in {\cal R}$ is two-sided Wiener process. This means that we
added an independent Wiener process $\tilde W_s,s\geq 0$ and put $\bar
W_t=\tilde W_{-t},t\leq 0 .$   Recall that the {\it innovation } stochastic
process $\left\{\bar{W}_t,\; 
t\geq 0\right\} $ is  defined by the equation
$$
\sigma \;\bar{W}_t=X_t-X_0-\int_{0}^{t}M\left(\vartheta_0,s\right){\rm d}s
$$
 is a standard Wiener process \cite{LS01}.

If $\vartheta \not=\vartheta _0$, then the process $m\left(\vartheta
,t\right),t\geq 0$ has the stochastic differential
\begin{align*}
{\rm d}m\left(\vartheta ,t\right)=-r\left(\vartheta \right)  m\left(\vartheta ,t\right){\rm
  d}t+ \Gamma \left(\vartheta \right)f\left(\vartheta \right)^{-1}f\left(\vartheta _0\right)m\left(\vartheta
_0,t\right){\rm d}t+  \Gamma \left(\vartheta \right)f^{-1} \sigma{\rm
  d}\bar W_t ,\; m\left(\vartheta ,0\right).
\end{align*}

 It will be more convenient to work with the stationary  stochastic process $
M(\vartheta ,t)=f(\vartheta )m(\vartheta,t )$ satisfying the equation
\begin{align*}
{\rm d}M\left(\vartheta ,t\right)=-r\left(\vartheta \right)  M\left(\vartheta ,t\right){\rm
  d}t+ \Gamma \left(\vartheta \right)M\left(\vartheta
_0,t\right){\rm d}t+  \Gamma \left(\vartheta \right) \sigma{\rm
  d}\bar W_t ,\quad M\left(\vartheta ,0\right).
\end{align*}

To write its representation in the integral form we consider two cases
separately. First we suppose that $r\left(\vartheta \right)\not=a
\left(\vartheta _0\right)$, then
\begin{align}
\label{2-13}
M\left(\vartheta ,t\right)&=\Gamma \left(\vartheta \right)\int_{-\infty
}^{t}e^{-r\left(\vartheta\right)\left(t-s\right)}M\left(\vartheta
_0,s\right){\rm d}s +\Gamma \left(\vartheta \right)\sigma \int_{-\infty
}^{t}e^{-r\left(\vartheta\right)\left(t-s\right)}{\rm d}\bar W_s\nonumber\\
&=\Gamma \left(\vartheta \right)\Gamma \left(\vartheta_0 \right)\sigma \int_{-\infty
}^{t}e^{-r\left(\vartheta\right)\left(t-s\right)}   \int_{-\infty
}^{s}e^{-a\left(\vartheta _0\right)\left(s-q\right)}{\rm d}\bar W_q  \; {\rm d}s\nonumber \\
&\qquad\qquad +\Gamma \left(\vartheta \right)\sigma \int_{-\infty
}^{t}e^{-r\left(\vartheta\right)\left(t-s\right)}{\rm d}\bar W_s\nonumber\\
&=\frac{\Gamma \left(\vartheta \right)\Gamma \left(\vartheta_0 \right)\sigma
}{r\left(\vartheta \right)-a\left(\vartheta _0\right)}\int_{-\infty
}^{t}e^{-a\left(\vartheta _0\right)\left(t-s\right)}{\rm 
  d}\bar W_s\nonumber\\
&\qquad\qquad +\frac{\Gamma \left(\vartheta \right)\sigma \left(r\left(\vartheta
  \right)-r\left(\vartheta_0 \right)\right)}{r\left(\vartheta
  \right)-a\left(\vartheta _0\right) } \int_{-\infty  
}^{t}e^{-r\left(\vartheta\right)\left(t-s\right)}{\rm d}\bar W_s.
\end{align}
Here we used the relation (Fubini theorem)
\begin{align*}
\int_{-\infty }^{t}e^{-r\left(\vartheta\right)\left(t-s\right)} \int_{-\infty
}^{s}e^{-a\left(\vartheta _0\right)\left(s-q\right)}{\rm d}\bar W_q \; {\rm d}s &
=e^{-r\left(\vartheta \right)t}\int_{-\infty }^{t}e^{a\left(\vartheta _0\right)q} \int_{q }^{t}e^{
  \left(r\left(\vartheta \right)-a\left(\vartheta _0\right)\right)s}{\rm d}s \; {\rm d} \bar W_q\\ &
= \frac{1}{r\left(\vartheta \right)-a\left(\vartheta _0\right)}\int_{-\infty
}^{t}\left[e^{-a\left(\vartheta _0\right)\left(t-q\right)}-e^{-r\left(\vartheta
    \right)\left(t-q\right)}\right] {\rm d} \bar W_q.
\end{align*}
This representation corresponds to the initial value
\begin{align*}
M\left(\vartheta ,0\right)&=\frac{\Gamma \left(\vartheta \right)\Gamma \left(\vartheta_0 \right)\sigma
}{r\left(\vartheta \right)-a\left(\vartheta _0\right)}\int_{-\infty }^{0}e^{a\left(\vartheta _0\right)s}{\rm
  d}\bar W_s +\frac{\Gamma \left(\vartheta \right)\sigma \left(r\left(\vartheta
  \right)-r\left(\vartheta_0 \right)\right)}{r\left(\vartheta \right)-a\left(\vartheta _0\right) }
\int_{-\infty 
}^{0}e^{r\left(\vartheta\right)s}{\rm d}\bar W_s.
\end{align*}

If $r\left(\vartheta \right)=a\left(\vartheta _0\right)$, then
\begin{align*}
\int_{-\infty }^{t}e^{-a\left(\vartheta _0\right)\left(t-s\right)} \int_{-\infty
}^{s}e^{-a\left(\vartheta _0\right)\left(s-q\right)}{\rm d}\bar W_q \; {\rm
  d}s&=e^{-a\left(\vartheta _0\right)t}\int_{-\infty }^{t} \int_{-\infty
}^{s}e^{a\left(\vartheta _0\right)q}{\rm d}\bar W_q \; {\rm d}s\\
&=\int_{-\infty }^{t}e^{-a\left(\vartheta _0\right)\left(t-s\right)}\left(t-s\right){\rm d}\bar W_s
\end{align*}
and
\begin{align*}
M\left(\vartheta ,t\right)=\Gamma \left(\vartheta \right)\sigma \int_{-\infty
}^{t}e^{-a\left(\vartheta _0\right)\left(t-s\right)}\left[ \gamma \left(\vartheta_0 \right)
  \left(t-s\right)+1\right]{\rm d}\bar W_s.
\end{align*}

These expressions will allow us calculate some moments involving
$M\left(\vartheta ,t\right)$ and $ \dot M\left(\vartheta ,t\right)$. The dot
means the derivation w.r.t. $\vartheta $.

{\bf MLE and BE}

We are given a two-dimensional diffusion process $\left\{X_t,Y_t,\;t\geq 0\right\}$
\begin{align*}
&{\rm d}Y_t=-\,a\left({ \vartheta}\right) Y_t\,{\rm
  d}t+b\left({  \vartheta}\right)\;{\rm d}V_t,\qquad Y_0,\;\\
&{\rm d}X_t=\; f\left({  \vartheta}\right)Y_t\,{\rm d}t+\;\sigma\, {\rm
d}W_t,\quad \quad X_0=0 ,
\end{align*}
where $\left\{V_t,W_t,\; t\geq 0\right\}$ are two independent Wiener
 processes, the functions $a\left(\cdot \right)$, $b\left(\cdot \right)$,
 $f\left(\cdot \right)$ and the constant $\sigma\neq 0 $ are known and ${  \vartheta}$ is a
 finite-dimensional (unknown) parameter. The random variable  $Y_0$
 is ${\goth F}_0$-measurable  and  Gaussian.

Suppose as usual that only one component $X^T=\left\{X_t,\;0\leq t\leq
T\right\}$ is observed and we have to estimate the parameter ${
  \vartheta}$. Our goal is to study the estimators of this parameter in the
asymptotic of large samples $T\rightarrow \infty $.

This model of observations allows us  to study estimators in the  five
different situations:
\begin{description}
\item[F.]  $f\left(\vartheta \right)=\vartheta,\quad  a\left(\vartheta
\right)=a,\quad  b\left(\vartheta \right)=b  $,

\item[A.]  $f\left(\vartheta \right)=f,\quad  a\left(\vartheta
\right)=\vartheta ,\quad  b\left(\vartheta \right)=b  $,

\item[B.] $f\left(\vartheta \right)=f,\quad  a\left(\vartheta
\right)=a ,\quad  b\left(\vartheta \right)=\vartheta  $,

\item[AB.] $f\left(\vartheta \right)=f,\quad  a\left(\vartheta
\right)=\theta _1  ,\quad  b\left(\vartheta \right)=\theta _2 $ , i.e.,
$\vartheta =\left(\theta _1,\theta _2\right)^\top$. 

\item[AF.] $f\left(\vartheta \right)=\theta _2,\quad  a\left(\vartheta
\right)=\theta _1 ,\quad  b\left(\vartheta \right)=b$.

\end{description}

 The MLE   $\hat{\vartheta }_\zT$ and BE $\tilde \vartheta _\zT$ are defined   by
the relations
\begin{align}
\label{2-MleBe}
L(\hat{\vartheta }_\zT,X^T
)=\sup_{\vartheta\in \Theta }L\left(\vartheta,X^T \right),\qquad \qquad \tilde\vartheta
_\zT=\frac{\int_{\Theta }^{}\vartheta p\left(\vartheta \right)L\left(\vartheta,X^T \right){\rm
    d}\vartheta}{\int_{\Theta }^{} p\left(\vartheta \right)L\left(\vartheta,X^T \right){\rm d}\vartheta}.
\end{align}
where
\begin{align}
\label{2-LR-T}
L\left(\vartheta,X^T \right)= \exp\left\{ \int_{0}^{T}\frac{M\left(\vartheta,s\right)}{\sigma ^2}\,{\rm d}X_s -
  \int_{0}^{t}\frac{M\left(\vartheta,s\right)^2}{2\sigma^2}{\rm d}s
  \right\} .
\end{align}
Here $M\left(\vartheta ,t\right)=f\left(\vartheta \right)m\left(\vartheta
,t\right)$  and  $p\left(\cdot \right)$ is continuous positive density function on 
$\Theta$.

To construct the MLE of the parameter $\vartheta $ we have to calculate the
random function
$\left\{M\left(\vartheta,t\right),\;0\leq t\leq T\right\},\;\vartheta\in
\Theta  $  defined 
by the equation
 \begin{align*}
{\rm d}M\left(\vartheta,t\right)=-r\left(\vartheta
\right)M\left(\vartheta,t\right)\,{\rm 
d}t+\Gamma \left(\vartheta \right)
 {\rm d}X_t ,\qquad \quad M\left(\vartheta,0 \right).
\end{align*}

It will be convenient to consider the cases ${\bf F, A, B}$ and ${\bf AB, AF}$
separately. Therefore we study first the estimation of one-dimensional
parameter and then we consider the cases ${\bf AB, AF}$ (two-dimensional parameter). 


 Suppose that the parameter $\vartheta $ is one-dimensional and $\Theta
 =\left(\alpha ,\beta \right)$.  Introduce the Fisher information
\begin{align*}
{\rm I}\left(\vartheta \right)=
 \frac{ \dot{a}\left(\vartheta \right)^2 }{2a\left(\vartheta \right)}-
\frac{2 \dot{a} \left(\vartheta \right)\;\dot{r} \left(\vartheta \right)
 }{r\left(\vartheta \right) +a \left(\vartheta \right)} +
 \frac{\dot{r}\left(\vartheta \right)^2 } {2r\left(\vartheta \right)},\qquad r\left(\vartheta \right)=\Bigl(a\left(\vartheta \right)^2+\sigma
  ^{-2}f\left(\vartheta \right)^2b\left(\vartheta \right)^2\Bigr)^{1/2}.
\end{align*}

Under regularity conditions given in the theorem below we have the lower bound
on the mean square risks of all estimators (Hajek-Le Cam's bound) \cite{IH81}
\begin{align*}
\lim_{\nu\rightarrow 0}\Liminf_{T\rightarrow \infty }\sup_{\left|\vartheta
  -\vartheta _0\right|\leq \nu }T\Ex_\vartheta \left|\bar\vartheta
_\zT-\vartheta \right|^2\geq {\rm I}\left(\vartheta_0 \right)^{-1}.
\end{align*}
We call the estimator $\bar\vartheta _\zT$ asymptotically efficient if for all
$\vartheta _0\in\Theta $ 
\begin{align*}
\lim_{\nu \rightarrow 0}\lim_{T\rightarrow \infty }\sup_{\left|\vartheta
  -\vartheta _0\right|\leq\nu }T\Ex_\vartheta \left|\bar\vartheta
_\zT-\vartheta \right|^2= {\rm I}\left(\vartheta_0 \right)^{-1}.
\end{align*}

\begin{theorem}
\label{T1}
Suppose that the following conditions hold.
\begin{enumerate}
\item The functions $a\left(\vartheta \right)$, $b\left(\vartheta\right)$,
  $f\left(\vartheta\right)$, $\vartheta \in \left[\alpha ,\beta \right]$ twice
   continuously differentiable on $\bar\Theta $.
\item   These  functions satisfy the conditions:
$b\left(\vartheta\right)\neq 0$, $f\left(\vartheta\right)\neq 0$ and
$a\left(\vartheta\right)>0$ for all $\vartheta \in
\left[\alpha ,\beta \right]$.
\item For all $\nu >0$ we have
\begin{align*}
\inf_{\vartheta_0 \in \Theta }\inf_{\left|\vartheta -\vartheta_0
\right|>\nu}\Bigl(\left|a(\vartheta)-a\left(\vartheta_0 \right)\right|+ \left|r
(\vartheta)-r \left(\vartheta_0 \right)\right|\Bigr)>0.
\end{align*}
\item The following condition holds
\begin{align*}
\inf_{\vartheta \in \Theta }\Bigl(\left|\dot{a}\left(\vartheta
\right)\right|+\left|\dot{r}\left(\vartheta \right)\right|\Bigr) >0.
\end{align*}
\end{enumerate}
 Then the MLE
$\hat{\vartheta }_\zT$ and BE $\tilde\vartheta _\zT$  are uniformly consistent,
 uniformly on compacts $\KK\subset\Theta $ asymptotically normal
\begin{align*}
\sqrt{T}\left(\hat{\vartheta
}_\zT-\vartheta_0  \right)\Longrightarrow \zeta \sim{\cal N}\left(0,{\rm
I}\left(\vartheta_0 \right)^{-1}\right),\quad  \sqrt{T}\left(\tilde{\vartheta
}_\zT-\vartheta_0  \right)\Longrightarrow \zeta ,
\end{align*}
polynomial moments converge: for any $p\geq 2$
\begin{align*}
\lim_{T\rightarrow \infty }T^{p/2}\Ex_{\vartheta _0}\left|\hat{\vartheta
}_\zT-\vartheta_0\right|^p=\Ex_{\vartheta _0}\left|\zeta\right|^p,\qquad
\lim_{T\rightarrow \infty }T^{p/2}\Ex_{\vartheta _0}\left|\tilde{\vartheta 
}_\zT-\vartheta_0\right|^p=\Ex_{\vartheta _0}\left|\zeta\right|^p,
\end{align*}
and the both estimators are   asymptotically efficient.
\end{theorem}
For the proof see Theorem 3.1 in \cite{Kut04}.

 {\bf Case {\bf F}$, i.e., \vartheta =f$.}

Consider the particular case $\vartheta =f$. Then $r\left(\vartheta
\right)=\left(a^2+{b^2\vartheta ^2}{\sigma ^{-2}}\right)^{1/2} $ and the
Fisher information is
\begin{align*}
{\rm I}\left(\vartheta \right)=\frac{b^4\vartheta ^2}{2\sigma
  ^4r\left(\vartheta \right)^3} .
\end{align*}
Hence by Theorem \ref{T1} the MLE and BE are consistent, asymptotically
normal
\begin{align*}
\sqrt{T}\left(\hat\vartheta _\zT-\vartheta _0\right)\Longrightarrow {\cal N}\left(0,
\frac{2\sigma ^4r\left(\vartheta _0 \right)^3}{b^4\vartheta _0^2}\right),\qquad \sqrt{T}\left(\tilde\vartheta _\zT-\vartheta _0\right)\Longrightarrow {\cal N}\left(0,
\frac{2\sigma ^4r\left(\vartheta _0 \right)^3}{b^4\vartheta _0^2}\right)
\end{align*}
and are asymptotically efficient.

 {\bf Case {\bf A}$, i.e., \vartheta =a$.}

We have $r\left(\vartheta
\right)=\left(\vartheta^2+{b^2f ^2}{\sigma ^{-2}}\right)^{1/2} $ and the
Fisher information is
\begin{align}
\label{2-43}
{\rm I}\left(\vartheta \right)=\frac{\left[r\left(\vartheta
    \right)^2-\vartheta ^2\right]^2+r\left(\vartheta \right)\vartheta
  \left[r\left(\vartheta \right)-\vartheta \right]^2 }{2\vartheta
  r\left(\vartheta \right)^3\left[r\left(\vartheta \right)+\vartheta \right]
}=\frac{\Gamma\left(\vartheta \right)^2 \left[\left(r\left(\vartheta \right)+\vartheta\right)^2
    +r\left(\vartheta \right) \vartheta 
\right]}{2\vartheta
  r\left(\vartheta \right)^3\left[r\left(\vartheta \right)+\vartheta \right]
},
\end{align}
 the both
estimators are asymptotically normal and asymptotically efficient. Say,
\begin{align*}
\sqrt{T}\left(\hat\vartheta _\zT-\vartheta _0\right)\Longrightarrow {\cal N}\left(0,
{\rm I}\left(\vartheta_0 \right)^{-1}\right).
\end{align*}

 {\bf Case {\bf B}$, i.e., \vartheta =b$.}

The Fisher information is
\begin{align*}
{\rm I}\left(\vartheta \right)=\frac{\vartheta ^2f^4}{2r\left(\vartheta
  \right)^3\sigma ^4} 
\end{align*}
and the MLE and BE have  the corresponding asymptotic properties.

\section{Method of moments estimators.}

Consider the  partially observed system
\begin{align}
\label{3-15}
{\rm d}X_t&=f\,Y_t\,{\rm d}t+\sigma \, {\rm
  d}W_t,\qquad\qquad  
\; X_0,\quad \qquad 0\leq t\leq T,\\
 {\rm d}Y_t&=-a\,Y_t\,{\rm d}t+b\,{\rm d}V_t,\qquad\qquad
Y_0,\qquad\quad t\geq 0,\label{3-16}
\end{align}

 The process $X_t,t\geq 0$ is not ergodic and that is why for construction of
 method of moments estimators we do not use integrals like
\begin{align*}
\frac{1}{T}\int_{0}^{T}g\left(X_t\right){\rm d}t.
\end{align*}

At the beginning we suppose that $\vartheta =\left(f,a,b\right)\in\Theta
$. The set $\Theta $ is a bounded, convex, closed subset of ${\cal R}^3$ and
the values of $a$ are positive and separated from zero. For simplicity of
exposition we suppose that $T$ is integer. The construction of estimators
follow the similar algorithms first used in the work \cite{KhK18}.

 Introduce the notation
\begin{align}
R_{1,T}&=\frac{1}{T}\sum_{k=1}^{T}\left[X_k-X_{k-1}\right]^2,\qquad \qquad
\quad \Phi_1 \left(\vartheta \right)=\frac{f^2b^2}{ a^3}\left[e^{-a }-1+a
  \right]+\sigma^2,\nonumber\\ 
R_{2,T}&=\frac{1}{T}\sum_{k=2}^{T}\left[X_k-X_{k-1}\right]\left[X_{k-1}-X_{k-2}\right],\qquad
\quad \Phi_2 \left(\vartheta \right)=\frac{f^2b^2}{2a^3}\left[1-e^{-a }
  \right]^2,\nonumber\\ 
K_{1,1} \left(\vartheta \right)&=\frac{{2f^4b^4}}{{{a} ^{6}}}\left[e^{-{a} }-1+{a}
    \right]^2+\frac{ f^4
 b^4e^{4a}\left[1-e^{-a}\right]^3}{a^6\left[1+e^{-a}\right]}\nonumber\\
&\qquad \qquad \qquad \qquad \qquad \qquad +\frac{4f^2\sigma^2b^2 
  }{a^3}\left[e^{-a}-1+a\right]  +{2\sigma^4   },\nonumber
\\ R_{T}&=\left(R_{1,T},R_{2,T}\right)^\top,\quad \qquad \Phi
\left(\vartheta \right)=\left(\Phi_1 \left(\vartheta \right),\Phi_2
\left(\vartheta \right)\right)^\top,\qquad \xi _*=\left(\xi _1,\xi
_2\right)^\top,\nonumber\\ 
{\bf K}\left(\vartheta \right)&=\left(
\begin{array}{cc}
K_{1,1} \left(\vartheta \right)&K_{1,2} \left(\vartheta \right)\nonumber\\
K_{2,1} \left(\vartheta \right)&K_{2,2} \left(\vartheta \right)\nonumber\\
\end{array}
\right),\quad R_{T}=\left(R_{1,T},R_{2,T}\right)^\top,\; \Phi
\left(\vartheta \right)=\left(\Phi_1 \left(\vartheta \right),\Phi_2
\left(\vartheta \right)\right)^\top.
\end{align}
Here $\xi _*\sim {\cal N}\left(0,{\bf K}\left(\vartheta_0 \right) \right)$ and
the calculation of other terms of the matrix ${\bf K}\left(\vartheta \right) $
we will  discuss later.

In the work \cite{Kut19b} the method of moments estimators (MME) $f_T^*,a_T^*$
and $b_T^*$ were proposed and studied for the one-dimensional parameters
$\vartheta =f$, $\vartheta =a$ and $\vartheta =b$ respectively. The MMEs are 
\begin{align*}
f_T^*&=\left(\frac{a^3}{b^2\left[e^{-a}-1+a\right]T}
\sum_{k=1}^{T}\left(\left[X_{k}-X_{k-1}\right]^2-\sigma^2\right)
\right)^{1/2},\\
b_T^*&=\left(\frac{a^3}{f^2\left[e^{-a}-1+a\right]T}
\sum_{k=1}^{T}\left(\left[X_{k}-X_{k-1}\right]^2-\sigma^2\right)
\right)^{1/2},\\
a_T^*&=H\left(\frac{1}{f^2b^2T}
\sum_{k=1}^{T}\left(\left[X_{k}-X_{k-1}\right]^2-\sigma^2\right)
\right),
\end{align*}
where $H\left(y\right)$ is a function inverse to $y=h\left(x\right)=
x^{-3}\left(x-1+e^{-x}\right)$. Below we show  that $h\left(x\right),x> 0$ is strictly
decreasing function.

It was shown in \cite{Kut19b} that these estimators are consistent and for any
$p\geq 1$
\begin{align}
\label{3-19}
\sup_{f\in \KK}T^{p/2}\Ex_{f}\left|f_T^*-f \right|^p\leq C,\;
\sup_{b\in \KK}T^{p/2}\Ex_{b}\left|b_T^*-b \right|^p\leq C, \;
\sup_{a\in \KK}T^{p/2}\Ex_{a}\left|a_T^*-a \right|^p\leq C, 
\end{align}
For construction of One-step MLE-process we need just the estimates \eqref{3-19}
only. 

The proposition below allows us to prove the asymptotic normality of these
estimators and to consider the two-dimensional case too.

\begin{proposition}
\label{3-P1}
 The statistic $R_\zT $  has the following properties:
\begin{enumerate}
\item Uniformly
on $ \Theta $ it converges to $\Phi \left(\vartheta \right) $,
i.e., for any $\nu >0$ 
\begin{align*}
\lim_{T \rightarrow \infty }\sup_{\vartheta \in\Theta}\Pb_\vartheta
\left(\left\|R_\zT -\Phi \left(\vartheta \right)\right\|\geq \nu \right)=0.
\end{align*}
\item Is uniformly
on compacts $\KK \subset\Theta 
$  asymptotically normal
\begin{align}
\label{3-20}
\sqrt{T }\left(R_\zT -\Phi \left(\vartheta \right)\right)\Longrightarrow  \xi _*,
\end{align}
\item The moments converge: for any $p>0$  uniformly on  $\KK 
$   
\begin{align*}
\lim_{T\rightarrow \infty }T ^{p/2}\Ex_\vartheta \left\|R_\zT -\Phi \left(\vartheta
\right)\right\|^p=\Ex_\vartheta \left|\xi _*\right|^p,
\end{align*}
and there exists a constant $C>0$ such that
\begin{align}
\label{3-12}
\sup_{\vartheta \in\KK}T ^{p/2}\Ex_\vartheta \left\|R_\zT -\Phi \left(\vartheta
\right)\right\|^p\leq  C.
\end{align}

\end{enumerate}
\end{proposition}
\begin{proof}

Using the equality \eqref{3-15}  we write
\begin{align*}
R_{1,T}&=\frac{1}{T}\sum_{k=1}^{T}\left[X_k-X_{k-1}\right]^2
=\frac{1}{T}\sum_{k=1}^{T}\left[\int_{k-1}^{k}{\rm
    d}X_s\right]^2\\
&=\frac{f^2}{T}\sum_{k=1}^{T}\eta _k^2 +\frac{2f\sigma
}{T}\sum_{k=1}^{T}\eta _k \left[W_k-W_{k-1}\right]+\frac{\sigma^2
}{T}\sum_{k=1}^{T} \left[W_
k-W_{k-1}\right]^2,
\end{align*}
where
\begin{align*}
\eta _k =\int_{k-1}^{k}Y_t\;{\rm d}t.
\end{align*}
We have
\begin{align*}
\Ex_{\vartheta _0}R_{1,T}&=\frac{f^2}{T}\sum_{k=1}^{T}\Ex_{\vartheta }\eta _k^2+\sigma ^2
\end{align*}
because $Y^T$ and $W^T$ are independent. 

For the  process $Y^T$ we have the equality 
\begin{align*}
Y_t=Y_0e^{-at}+b\int_{0}^{t}e^{-a\left(t-s\right)}{\rm d}V_s.
\end{align*} 
Hence
\begin{align*}
\Ex_{\vartheta}Y_tY_s&=\Ex_{\vartheta _0}Y_0 ^2e^{-{a}
  \left(t+s\right)}+b^2e^{-{a} \left(t+s\right)}
\int_{0}^{t\wedge s}e^{2{a} r}{\rm d}r\\
&=\left[d^2-\frac{b^2}{2a} \right] e^{-{a}
  \left(t+s\right)}+\frac{b^2}{2a}e^{-{a}
  \left|t-s\right|}
\end{align*}
and
\begin{align*}
\Ex_{\vartheta}\eta _k^2& =\int_{k-1}^{k}\int_{k-1}^{k}\Ex_{\vartheta
}Y_tY_s\,{\rm d}s{\rm d}t \\ &=\left[d^2-\frac{b^2}{2a}
  \right]\left(\int_{k-1}^{k}e^{-{a} t}{\rm d}t\right)^2
+\frac{b^2}{2a}\int_{k-1}^{k}\int_{k-1}^{k}e^{-{a} \left|t-s\right|}{\rm
  d}s{\rm d}t\\ &=\left[\frac{d^2}{a^2}-\frac{b^2}{2a^3} \right] \left[e^{a}
  -1\right]^2e^{-2{a} k}+\frac{b^2}{{a} ^3}\left[e^{-{a} }-1+{a} \right].
\end{align*}
Therefore
\begin{align*}
\Ex_{\vartheta}R_{1,T}&=\left[e^{a}
  -1\right]^2\frac{f^2\left[2ad^2 -b^2\right]}{{2a^3}
  T}\sum_{k=1}^{T}e^{-2{a} k}+\frac{f^2b^2}{{a}
  ^3}\left[e^{-{a} }-1+{a} \right]+\sigma ^2 \\
&=\frac{f^2b^2}{{a}
  ^3}\left[e^{-{a} }-1+{a} \right]+\sigma ^2+r_\zT ,\qquad
\left|r_\zT\right|\leq \frac{C}{T}. 
\end{align*}
For simplicity of calculations we suppose that the process $Y_t,t\geq 0$ is
stationary. Then 
\begin{align*}
\Ex_{\vartheta}R_{1,T}&=\frac{f^2b^2}{{a}
  ^3}\left[e^{-{a} }-1+{a} \right]+\sigma ^2.
\end{align*}
 
The covariance is ($j<k$)
\begin{align*}
\Ex_{\vartheta }\eta _k\eta _j&=b^2\int_{k-1}^{k}\int_{j-1}^{j}\Ex_{\vartheta
} Y_tY_s {\rm d}s{\rm
  d}t=b^2\int_{k-1}^{k}\int_{j-1}^{j}e^{-a\left(t+s\right)}\int_{-\infty
}^{t\wedge s}e^{2au}{\rm d}u {\rm d}s{\rm d}t\\
&=\frac{b^2}{2a}\int_{k-1}^{k}\int_{j-1}^{j}e^{-a\left|t-s\right|}{\rm d}s{\rm
  d}t=\frac{b^2e^a}{a^3}\left(1-e^{-a}\right)^2 \,e^{-a\left|k-j\right|}.
\end{align*}
Hence 
\begin{align}
\label{3-13}
\sup_{\vartheta \in\Theta }\Ex_{\vartheta }\eta _k\eta _j\leq C \,e^{-a\left|k-j\right|}
\end{align}
and this Gaussian sequence $\eta _1,\eta _2,\ldots $ is exponentially mixing.
By the law of large numbers
\begin{align*}
 R_{1,T}\longrightarrow \Phi_1 \left(\vartheta \right)=\frac{f^2b^2}{{a}
   ^3}\left[e^{-{a} }-1+{a} \right]+\sigma ^2,
\end{align*}
as $T\rightarrow \infty $.

 Then we can write
\begin{align*}
&\Ex_{\vartheta }\left(R_{1,T}-\Ex_{\vartheta}R_{1,T} \right)^2\\
&\qquad =\Ex_{\vartheta
  }\left(\frac{f^2}{T}\sum_{k=1}^{T}\left[\eta _k^2- \Ex_{\vartheta}\eta _k^2
    \right] +\frac{2f\sigma }{T}\sum_{k=1}^{T}\eta _k \zeta _k+\frac{\sigma^2
  }{T}\sum_{k=1}^{T} \left(\zeta _k^2-1\right)\right)^2\\
 &\qquad
  =\frac{f^4}{T^2}\sum_{k=1}^{T}\sum_{j=1}^{T}\Ex_{\vartheta }\left(\left[\eta
    _k^2- \Ex_{\vartheta}\eta _k^2 \right]\left[\eta _j^2- \Ex_{\vartheta}\eta
    _j^2 \right]\right)+\frac{4f^2\sigma^2
  }{T^2}\sum_{k=1}^{T}\sum_{j=1}^{T}\Ex_{\vartheta }\bigl(\eta _k \eta _j
  \zeta _k\zeta _j\bigr)\\
&\qquad \quad +\frac{\sigma^4
  }{T^2}\sum_{k=1}^{T}\sum_{j=1}^{T}\Ex_{\vartheta } \left(\zeta _k^2-1\right)\left(\zeta _j^2-1\right)\\
 &\qquad
  =\frac{f^4}{T^2}\sum_{k=1}^{T}\sum_{j=1}^{T}\Ex_{\vartheta }\left(\left[\eta
    _k^2- \Ex_{\vartheta}\eta _k^2 \right]\left[\eta _j^2- \Ex_{\vartheta}\eta
    _j^2 \right]\right)+\frac{4f^2\sigma^2
  }{T^2}\sum_{k=1}^{T}\Ex_{\vartheta } \eta _k^2 \\
&\qquad \quad +\frac{\sigma^4
  }{T^2}\sum_{k=1}^{T}\Ex_{\vartheta } \left(\zeta _k^2-1\right)^2\\
 &\qquad
  \leq \frac{f^4}{T^2}\sum_{k=1}^{T}\sum_{j=1}^{T}\Ex_{\vartheta }\left(\left[\eta
    _k^2- \Ex_{\vartheta}\eta _k^2 \right]\left[\eta _j^2- \Ex_{\vartheta}\eta
    _j^2 \right]\right)+\frac{4f^2\sigma^2b^2
  }{a^3T}\left[a-1+e^{-a}\right]  +\frac{2\sigma^4
  }{T}     .
\end{align*}
For the last sum using once more the presentation \eqref{3-16} we obtain the
exponential bound like \eqref{3-13} as follows. Note that 
\begin{align*}
\Ex_{\vartheta }\left(\left[\eta
    _k^2- \Ex_{\vartheta}\eta _k^2 \right]\left[\eta _j^2- \Ex_{\vartheta}\eta
    _j^2 \right]\right)=\Ex_{\vartheta }\left[\eta
    _k^2\eta _j^2\right]-\left(\Ex_{\vartheta }\left(\eta
    _1^2\right)\right)^2,
\end{align*}
where $\Ex_{\vartheta }\eta _1^2={b^2}{{a} ^{-3}}\left[e^{-{a} }-1+{a}
  \right] $.

We have 
\begin{align*}
\Ex_{\vartheta }\left[\eta _k^2\eta
  _j^2\right]&=\int_{k-1}^{k}\int_{k-1}^{k}\int_{j-1}^{j}\int_{j-1}^{j}
\Ex_\vartheta Y_tY_sY_rY_q\;{\rm d}t{\rm d}s{\rm d}r{\rm d}q \\
&=
b^4\int_{k-1}^{k}\int_{k-1}^{k}\int_{j-1}^{j}\int_{j-1}^{j}e^{-a\left(t+s+r+q\right)}\Ex_\vartheta
I\left(t \right) I\left(s \right) I\left(r \right) I\left(q \right) \;{\rm
  d}t{\rm d}s{\rm d}r{\rm d}q .
\end{align*}
Suppose that $j\leq 
k-1$. Then using the elementary properties of the stochastic
integral we can write  
\begin{align*}
\Ex_\vartheta
I\left(t \right) I\left(s \right) I\left(r \right) I\left(q
\right)&=\Ex_\vartheta\Bigl(  I\left(r \right) I\left(q \right)\Ex_\vartheta 
\Bigl( I\left(t \right) I\left(s \right)|{\scr F}_{r\vee q}^Y\Bigr)\Bigr)\\
&=\Ex_\vartheta\Bigl(  I\left(r \right) I\left(q \right)\Ex_\vartheta 
\Bigl(  I\left(s\wedge t \right)^2|{\scr F}_{r\vee q}^Y\Bigr)\Bigr)\\
&=\Ex_\vartheta\Bigl(  I\left(r \right) I\left(q \right)\Ex_\vartheta 
\Bigl(  \left[I\left(s\wedge t \right)-I\left(r\vee q\right)+ I\left(r\vee
  q\right)\right]^2|{\scr F}_{r\vee q}^Y\Bigr)\Bigr) \\
&=\Ex_\vartheta\Bigl(  I\left(r \right) I\left(q \right)
\Bigl( \int_{r\vee q}^{s\wedge t} e^{2au}{\rm d}u+ I\left(r\vee
  q\right)^2\Bigr)\Bigr) \\
&=\frac{1}{4a^2}\left[e^{2a \left(s\wedge t\right)}-e^{2a \left(r\vee
      q\right)}\right]e^{2a \left(r\wedge
      q\right)} +
  \Ex_\vartheta\Bigl(  I\left(r\wedge q \right)
  I\left(r\vee   q\right)^3\Bigr) \\
&=\frac{1}{4a^2}\left[e^{2a\left[ \left(s\wedge t\right)+ \left(r\wedge
      q\right) \right]}-e^{2a \left(r+
      q\right)}\right]+\frac{3}{4a^2} e^{2a \left(r+
      q\right)}\\
&=\frac{1}{4a^2}\left[e^{2a\left[ \left(s\wedge t\right)+ \left(r\wedge
      q\right) \right]}+2e^{2a \left(r+
      q\right)}\right].
\end{align*}
Here we denoted  $\sigma $-algebra ${\scr F}_t^Y=\sigma \left(Y_s,s\leq t\right)
$ and  used the equality
\begin{align*}
& \Ex_\vartheta\Bigl( I\left(r\wedge q \right) I\left(r\vee q\right)^3\Bigr)=
 \Ex_\vartheta\Bigl( I\left(r\wedge q \right) \left[ I\left(r\vee
   q\right)-I\left(r\wedge q \right)+I\left(r\wedge q \right)\right] ^3\Bigr)
 \\ & \qquad \qquad = 3\Ex_\vartheta\Bigl( I\left(r\wedge q \right)\Bigr)^2
 \Ex_\vartheta\Bigl( \left[I\left(r\vee q\right)-I\left(r\wedge
   q\right)\right]^2\Bigr)+\Ex_\vartheta\Bigl( I\left(r\wedge q \right)\Bigr)
 ^4 \\ & \qquad \qquad = \frac{3}{4a^2}e^{2a\left( r\wedge
   q\right)}\left[e^{2a\left( r\vee q\right)}- e^{2a\left( r\wedge
     q\right)}\right]+\frac{3}{4a^2}e^{4a\left( r\wedge q\right)}
 =\frac{3}{4a^2}e^{2a\left( r+   q\right)} .
\end{align*}
Therefore
\begin{align*}
e^{-a\left(t+s+r+q\right)}\Ex_\vartheta
I\left(t \right) I\left(s \right) I\left(r \right) I\left(q
\right)=\frac{1}{4a^2} e^{-a\left|t-s\right|-a|r-q|}+\frac{1}{2a^2}e^{-a
  \left(t+s\right)  }e^{a
  \left(r+q\right)  }
\end{align*}
and
\begin{align*}
&\int_{k-1}^{k}\int_{k-1}^{k}\int_{j-1}^{j}\int_{j-1}^{j}e^{-a\left(t+s+r+q\right)}\Ex_\vartheta
  I\left(t \right) I\left(s \right) I\left(r \right) I\left(q \right) \;{\rm
    d}t{\rm d}s{\rm d}r{\rm d}q \\
 &\qquad \qquad
  =\frac{1}{4a^2}\int_{k-1}^{k}\int_{k-1}^{k}\int_{j-1}^{j}\int_{j-1}^{j}
  e^{-a\left|t-s\right|-a|r-q|}\;{\rm d}t{\rm d}s{\rm d}r{\rm d}q \\ 
&\qquad
  \qquad\qquad +\frac{1}{2a^2}
  \int_{k-1}^{k}\int_{k-1}^{k}\int_{j-1}^{j}\int_{j-1}^{j} e^{-a
    \left(t+s-r-q\right) }\;{\rm d}t{\rm d}s{\rm d}r{\rm d}q\\ 
&\qquad \qquad   =\frac{{1}}{{{a} ^{6}}}\left[e^{-{a} }-1+{a}
    \right]^2+\frac{e^{2a}\left[1-e^{-a}\right]^4}{2a^6}\;e^{-2a\left|k-j\right|} .
\end{align*}
Hence if $j\leq k-1$ then
\begin{align*}
 \Ex_{\vartheta }\left(\left[\eta
    _k^2- \Ex_{\vartheta}\eta _k^2 \right]\left[\eta _j^2- \Ex_{\vartheta}\eta
    _j^2
   \right]\right)=\frac{b^4e^{2a}\left[1-e^{-a}\right]^4}{2a^6}\;e^{-2a\left|k-j\right|}
 \leq  C\;e^{-2a\left|k-j\right|}.
\end{align*}
If $k=j$ then 
\begin{align*}
\Ex_{\vartheta }\eta
    _k^4=3\left( \Ex_{\vartheta }\eta
    _k^2\right)^2 =\frac{{3b^4}}{{{a} ^{6}}}\left[e^{-{a} }-1+{a}
    \right]^2
\end{align*}
and
\begin{align*}
\Ex_{\vartheta }\left[\eta
    _k^2- \Ex_{\vartheta}\eta _k^2 \right]^2=\frac{{2b^4}}{{{a} ^{6}}}\left[e^{-{a} }-1+{a}
    \right]^2.
\end{align*}
Therefore
\begin{align*}
&\frac{1}{T}\sum_{k=1}^{T}\sum_{j=1}^{T}\Ex_{\vartheta }\left(\left[\eta _k^2-
    \Ex_{\vartheta}\eta _k^2 \right]\left[\eta _j^2- \Ex_{\vartheta}\eta _j^2
    \right]\right)\\
 &\qquad \quad \quad =\frac{{2b^4}}{{{a}
      ^{6}}}\left[e^{-{a} }-1+{a}
    \right]^2+\frac{b^4e^{2a}\left[1-e^{-a}\right]^4}{2a^6T}\sum_{k=1}^{T}\sum_{j\not=k}^{T}
  e^{-2a\left|k-j\right|} \\ 
&\qquad \quad \quad \longrightarrow
  \frac{{2b^4}}{{{a} ^{6}}}\left[e^{-{a} }-1+{a} \right]^2+\frac{
    b^4e^{4a}\left[1-e^{-a}\right]^3 }{a^6\left[1+e^{- a}\right]}.
\end{align*}
Here we used the relations
\begin{align*}
\sum_{j=1,j\not=k}^{T}e^{-2a\left|k-j\right|}=\frac{2e^{2a}}{1-e^{-2a} }-\frac{e^{-2ak}+e^{-2a\left(T-k+1\right)}}{1-e^{-2a} }
\end{align*}
and
\begin{align*}
\frac{\left[1-e^{-a}\right]^4}{\left(1-e^{-a}\right)\left(1+e^{-a}\right)}=\frac{\left[1-e^{-a}\right]^3}{\left(1+e^{-a}\right)}.
\end{align*}
Therefore we obtained the estimate
\begin{align*}
\sup_{\vartheta \in\KK}\Ex_{\vartheta }\Bigl(R_{1,T}-\Ex_{\vartheta}R_{1,T}
\Bigr)^2\leq \frac{C}{T} 
\end{align*}
and the convergence
\begin{align*}
{T}\Ex_{\vartheta }\Bigl(R_{1,T}-\Ex_{\vartheta}R_{1,T}\Bigr)^2&\longrightarrow
\frac{{2f^4b^4}}{{{a} ^{6}}}\left[e^{-{a} }-1+{a}
    \right]^2+\frac{
 f^4 b^4e^{4a}\left[1-e^{-a}\right]^3}{a^6\left[1+e^{-a}\right]}\\
&\qquad +\frac{4f^2\sigma^2b^2 
  }{a^3}\left[e^{-a}-1+a\right]  +{2\sigma^4   }=K_{1,1}\left(\vartheta \right)   .
\end{align*}
The estimate \eqref{3-12} is obtained directly from the Rosenthal's
inequality for time series. As the random variables
$\left|\sqrt{T}\left(R_{1,T}-\Phi _1\left(\vartheta \right)\right)\right|^p $
are uniformly integrable we have the convergence of polynomial moments too.

By the central limit theorem for exponentially mixing time series we obtain
the convergence
\begin{align*}
\sqrt{T}\left(R_{1,T}-\Phi _1\left(\vartheta \right)\right)\Longrightarrow \xi _1\sim
     {\cal N}\left(0, K_{1,1}\left(\vartheta \right)\right). 
\end{align*}

Consider now the second component
\begin{align*}
 R_{2,T}&=\frac{1}{T}\sum_{k=2}^{T}\left[f\eta _k+\sigma
  \left(W_k-W_{k-1}\right)\right]\left[f\eta _{k-1}+\sigma
  \left(W_{k-1}-W_{k-2}\right)\right]\\
&= \frac{f^2}{T}\sum_{k=2}^{T}\eta _k\eta _{k-1}+ \frac{f\sigma
}{T}\sum_{k=2}^{T}\left[\eta _k\zeta _{k-1}+ \eta _{k-1}\zeta _{k}\right]  +\frac{\sigma
  ^2}{T}\sum_{k=2}^{T}\zeta _{k}\zeta _{k-1} .
\end{align*}
Recall  that
\begin{align*}
\Ex_\vartheta \eta _k\zeta _{k-1}=0,\qquad \Ex_\vartheta\eta _{k-1}\zeta
_{k}=0,\qquad \Ex_\vartheta\zeta _{k}\zeta  _{k-1}=0 
\end{align*}
and by the law of large numbers we have the convergences
\begin{align*}
\frac{1 }{T}\sum_{k=2}^{T}\eta _k\zeta _{k-1}\longrightarrow 0,\qquad\frac{1
}{T}\sum_{k=2}^{T}\eta _{k-1}\zeta _{k}\rightarrow 0,\qquad 
 \frac{1}{T}\sum_{k=2}^{T}\zeta _{k}\zeta _{k-1}\longrightarrow 0.
\end{align*}
Further, 
\begin{align*}
\Ex_{\vartheta}R_{2,T}&= \frac{f^2}{T}\sum_{k=2}^{T}\Ex_{\vartheta}\eta _k\eta
_{k-1}=\frac{f^2b^2}{2aT}\sum_{k=2}^{T}
\int_{k-1}^{k}\left(\int_{k-2}^{k-1}e^{-a\left|t-s\right|}{\rm d}s\right) {\rm
  d}t\\
&\longrightarrow \frac{f^2b^2}{2a^3}\left[1-e^{-a}\right]^2\equiv \Phi _2\left(\vartheta \right).
\end{align*}
For the variance  of $R_{2,T} $ we have the following expression
\begin{align*}
\Ex_\vartheta \left(R_{2,T}-\Ex_\vartheta R_{2,T}\right)^2&=
\frac{f^4}{T^2}\sum_{k=2}^{T}\sum_{j=2}^{T}\Ex_\vartheta\left[\eta _k\eta
  _{k-1}-\Ex_\vartheta\eta _k\eta _{k-1}\right]\left[\eta _j\eta
  _{j-1}-\Ex_\vartheta\eta _j\eta _{j-1}\right]\\
 &\qquad + \frac{f^2\sigma^2
}{T^2}\sum_{k=2}^{T}\sum_{j=2}^{T}\Ex_\vartheta\left[\eta _k\zeta _{k-1}+ \eta
  _{k-1}\zeta _{k}\right]\left[\eta _j\zeta _{j-1}+ \eta _{j-1}\zeta
  _{j}\right]\\
 &\qquad +\frac{\sigma
  ^4}{T^2}\sum_{k=2}^{T}\sum_{j=2}^{T}\Ex_\vartheta\zeta _{k}\zeta _{k-1}\zeta
_{j}\zeta
_{j-1}\\ &=\frac{f^4}{T^2}\sum_{k=2}^{T}\sum_{j=2}^{T}\left(\Ex_\vartheta\left[\eta
  _k\eta _{k-1}\eta _j\eta _{j-1}\right]-\left[\Ex_\vartheta\eta _k\eta
  _{k-1}\right]^2\right)+ \frac{\sigma
  ^4\left(T-2\right)}{T^2}\\ &\qquad +\frac{f^2\sigma^2
}{T^2}\sum_{k=2}^{T}\Ex_\vartheta\left[\eta _k^2+2\eta _k \eta _{k-2}+ \eta
  _{k-1}^2\right]. 
\end{align*}
The values of these expectations are already calculated above
\begin{align*}
f^2\Ex_\vartheta\eta _k^2=\Phi _1\left(\vartheta \right),\quad
f^2\Ex_\vartheta\eta _k\eta _{k-1}= \Phi _2\left(\vartheta \right),\quad
f^2\Ex_\vartheta\eta _k\eta _{k-2}= e^{-a}\Phi _2\left(\vartheta \right).
\end{align*}
The calculation of the expectation $\Ex_\vartheta\left[\eta _k\eta _{k-1}\eta
  _j\eta _{j-1}\right] $ can be carried out following the same lines as it was
done above.  Hence as above we can obtain the estimate
\begin{align*}
\sup_{\vartheta \in\KK} \Ex_\vartheta \left(R_{2,T}-\Ex_\vartheta R_{2,T}\right)^2\leq \frac{C}{T}.
\end{align*}
By  the law of large numbers 
\begin{align*}
R_{2,T}\longrightarrow \Phi _2\left(\vartheta \right)
\end{align*}
and by the central limit theorem
\begin{align*}
\sqrt{T}\left(R_{2,T}-\Phi _2\left(\vartheta \right) \right)\Longrightarrow \xi _2.
\end{align*}
Moreover it can be verified by the direct calculations that for any $\lambda
,\mu $ 
\begin{align*}
\sqrt{T}\Bigl(\lambda \left(R_{1,T}-\Phi _1\left(\vartheta \right)\right)+\mu
\left(R_{2,T}-\Phi _2\left(\vartheta \right)\right)  \Bigr)\Longrightarrow
\lambda \xi _1+\mu \xi _2. 
\end{align*}
and this provides the asymptotic normality \eqref{3-20}.
 These calculations are direct, are close to the given above calculations of $
 K_{1,1}\left(\vartheta \right)$ and cumbersome that is why we omit
them and omit as well the corresponding calculations related with the terms
$K_{1,2}\left(\vartheta \right)=K_{2,1}\left(\vartheta \right)$ and
$K_{2,2}\left(\vartheta \right) $ of the matrix ${\bf K}\left(\vartheta \right)$.

Remark that for any $\nu >0$
\begin{align*}
\sup_{\vartheta \in\KK}\Pb_\vartheta \left(\left\| R_{T} -\Phi\left(\vartheta
\right) \right\|\geq \nu \right)\leq {\nu ^{-2}}{\sup_{\vartheta
    \in\KK}\Ex_\vartheta \left\|R_{T} -\Phi\left(\vartheta \right)
  \right\|^2}\leq \frac{C}{\nu ^2T}
\end{align*}

\end{proof}

{\bf Estimation of the parameters $f,b$ and $a$.}

The asymptotic normality of the estimators $f^*_T,b^*_T,a^*_T$ we obtain from
Proposition \ref{3-P1} as follows.  For the estimators $f^*_T $ and $b^*_T $ we have
\begin{align*}
\sqrt{T}\left(f^*_T-f _0 \right)&\Rightarrow \xi _f\sim {\cal
  N}\left(0,D_f\left(f_0\right)^2\right) ,\quad  D_f\left(f \right)^2=\frac{a^3\;K_{1,1}\left(f
  \right)}{4b^2\left(e^{-a}-1+a\right)\left(\Phi _1\left(f
  \right)-\sigma ^2\right) },\\
\sqrt{T}\left(b_\zT^*-b _0 \right)&\Rightarrow \xi _b\sim {\cal
  N}\left(0,D_b\left(b_0\right)^2\right) ,\quad D_b\left(b \right)^2=\frac{a^3\;K_{1,1}\left(b
  \right)}{4f^2\left(e^{-a}-1+a\right)\left(\Phi _1\left(b
  \right)-\sigma ^2\right) }.
\end{align*}

The asymptotic normality of $a^*_T $ is obtained as follows. 
Suppose that $\vartheta =a\in \Theta =\left(\alpha ,\beta \right), \alpha >0$
and the values of  $f\not=0$ and  $b\not=0$ are known. 

First we  verify that the function $\Phi _1\left(a \right),a
\in\Theta $ is strictly decreasing.  Introduce the function $h
\left(x\right)=x^{-3}\left(x-1+e^{-x}\right),x>0 $. Then for its derivative we
have the expression $h '\left(x\right)=x^{-4}\left[3-2x-\left(3+x\right)e^{-x}
  \right] $. Remark, that the function
$g\left(x\right)=3-2x-\left(3+x\right)e^{-x} $ at $x=0$ is $g\left(0\right)=0$
and has derivative $g'\left(x\right)=-2+\left(2+x\right)e^{-x}$ with
$g'\left(0\right)=0$. At last $g''\left(x\right)=-\left(1+x\right)e^{-x}< 0
$. Hence $g'\left(x\right)< 0 $, $g\left(x\right)< 0 $ and $ h
'\left(x\right)< 0 $ for $x>0$. As the function $\Phi _1\left(a
\right),a \in\Theta  $ is strictly decreasing, the equation 
\begin{align*}
R_{1,T}=\Phi _1\left(a\right),\qquad a \in\Theta 
\end{align*}
has a unique solution $a =a_\zT^*$ for the values $\Phi _1\left(\beta
\right) < R_{1,T} < \Phi _1\left(\alpha \right)$.
\begin{align*}
a _\zT^*=H\left(\frac{R_{1,T}-\sigma
  ^2}{f^2b^2}\right)&=H\left(\frac{\Phi _1\left(a _0 \right)-\sigma
  ^2}{f^2b^2}\right)+H'\left(\frac{\Phi _1\left(a_0 \right)-\sigma
  ^2}{f^2b^2}\right)\frac{\left(R_{1,T}-\Phi _1\left(a _0
  \right)\right)}{f^2b^2} \\
&=a _0+\frac{R_{1,T}-\Phi _1\left(a _0
  \right)}{h'\left(a _0\right)f^2b^2}+O\left(T^{-1}\right).
\end{align*}
Hence
\begin{align*}
\sqrt{T}\left(a _\zT^*-a _0\right)\Longrightarrow \xi _a\sim {\cal
  N}\left(0,D_a\left(a_0\right)^2\right),\qquad \quad D_a\left(a
_0\right)^2=\frac{K_{1,1}\left(a _0\right)}{h'\left(a_0\right)^2f^4b^4}.
\end{align*}

As the derivative $h'\left(x\right),x>\alpha $ is separated from zero, we have
the estimate: for any $p\geq 2$ there exists a constant $C>0$ such that
\begin{align*}
\sup_{a_0\in\KK}T^{p/2}\Ex_{a_0}\left|a
_\zT^*-a _0\right|^p\leq C 
\end{align*}
and 
the convergence of moments
\begin{align*}
\lim_{T\rightarrow \infty }T^{p/2}\Ex_{a_0}\left|a
_\zT^*-a_0\right|^p=\Ex_{a _0}\left|\xi _a\right|^p.
\end{align*}
holds. 

{\bf More general model}

All considered cases of estimation of  one-dimensional parameters are particular
cases of a slightly more general model of observations 
\begin{align*}
{\rm d}X_t&=f\left(\vartheta \right)Y_t{\rm d}t+\sigma {\rm d}W_t,\qquad \qquad
X_0,\qquad 0\leq t\leq T,\\
{\rm d}Y_t&=-a\left(\vartheta \right)Y_t{\rm d}t+b\left(\vartheta \right){\rm d}V_t,\qquad \quad Y_0,\quad t\geq 0
\end{align*}
where $f\left(\vartheta \right),a\left(\vartheta \right),b\left(\vartheta
\right),\vartheta \in\Theta =\left(\alpha ,\beta \right)$ are known smooth
functions. To estimate the parameter $\vartheta $ by observations
$X^T=\left(X_t,0\leq t\leq T\right)$ we can use the MME
\begin{align}
\label{3-21a} 
R_{1,T}&=\frac{1}{T}\sum_{k=1}^{T}\left[X_k-X_{k-1}\right]^2,\qquad \quad
\Psi\left(\vartheta \right)=\frac{f\left(\vartheta \right)^2b\left(\vartheta
  \right)^2}{a\left(\vartheta \right)^2}\left[e^{-a\left(\vartheta
    \right)}-1+a\left(\vartheta \right)\right],\nonumber\\
 \vartheta^* _T&={\rm arg}\inf_{\vartheta \in\Theta
}\left|R_{1,T}-\Psi\left(\vartheta \right)\right|,
\end{align}

\begin{proposition}
\label{P4} Let the functions  $f\left(\vartheta \right),a\left(\vartheta \right),b\left(\vartheta
\right),\vartheta \in\left[\alpha ,\beta \right]$  have two continuous
derivatives, and 
\begin{align*}
\inf_{\vartheta \in\Theta }\left|\dot \Psi\left(\vartheta \right)\right|>0.
\end{align*}
Then the MME $\check \vartheta _T$ is uniformly consistent and for any  $p\geq 2$
\begin{align}
\label{3-21} 
\sup_{\vartheta _0\in\KK}T^{p/2}\Ex_{\vartheta _0}\left|\check \vartheta _T-\vartheta _0 \right|^p\leq C.
\end{align}

\end{proposition}
\begin{proof} Note that for any $\nu >0$
\begin{align*}
g\left(\vartheta _0,\nu \right)&=\inf_{\left|\vartheta -\vartheta _0\right|>\nu
}\left|\Psi\left(\vartheta \right)-\Psi\left(\vartheta_0 \right) \right|\geq
\inf_{\left|\vartheta -\vartheta _0\right|>\nu } \left|\vartheta -\vartheta _0
\right| \left|\dot \Psi \left(\bar\vartheta \right)\right| \\
&\geq \nu \inf_{ \vartheta  \in \Theta  }\left|\dot \Psi \left(\vartheta
\right)\right|\geq  \check\kappa\; \nu ,
\end{align*}
where  $\check \kappa >0$. Then 
the consistency follows from the standard relations
\begin{align*}
\Pb_{\theta _0}\left(\left|\check \vartheta _T-\vartheta _0 \right|\geq  \nu
\right)&=\Pb_{\theta _0}\left(\inf_{ \left| \vartheta -\vartheta _0 \right|\leq
  \nu} \left| R_{1,T}-\Psi \left(\vartheta \right)\right| \geq \inf_{ \left|
  \vartheta -\vartheta _0 \right|>\nu} \left| R_{1,T}-\Psi \left(\vartheta
\right)\right| \right)\\
& \leq \Pb_{\theta _0}\left(2\left| R_{1,T}-\Psi \left(\vartheta_0
\right)\right| \geq \inf_{ \left| 
  \vartheta -\vartheta _0 \right|>\nu} \left| \Psi \left(\vartheta
\right) -\Psi \left(\vartheta_0
\right)\right| \right)\\
&\leq \frac{4}{g\left(\vartheta _0,\nu \right)^2  }\Ex_{\vartheta
  _0}\left|\check \vartheta _T-\vartheta _0 \right|^2 \leq
\frac{C}{Tg\left(\vartheta _0,\nu \right)^2 }\rightarrow 0. 
\end{align*}

The estimator $\check \vartheta _T $ with probability tending to 1 is
solution of the equation 
\begin{align*}
0=R_{1,T}-\Psi \left(\check \vartheta _T \right)=R_{1,T}-\Psi \left( \vartheta
_0 \right)-\left(\check \vartheta _T-\vartheta _0 \right)\dot \Psi \left(\bar
\vartheta _T \right) 
\end{align*}
where $\left|\bar\vartheta_\zT -\vartheta _0\right|\leq \left|\check\vartheta_\zT -\vartheta _0\right|$   and therefore
\begin{align*}
\sqrt{T}\left(\check \vartheta _T-\vartheta _0
\right)=\frac{\sqrt{T}\left(R_{1,T}-\Psi \left( \vartheta 
_0 \right)\right)  }{\dot \Psi \left(\bar \vartheta _T \right)  }.
\end{align*}
Hence (see \eqref{3-12})
\begin{align*}
T^{p/2}\Ex_{\vartheta _0}\left|\check \vartheta _T-\vartheta _0  \right|^p\leq
\frac{T^{p/2}\Ex_{\vartheta _0} \left| R_{1,T}-\Psi \left( \vartheta
_0 \right)   \right|^p    }{\inf_{\vartheta \in\Theta }\left|\dot
  \Psi\left(\vartheta \right)\right|}\leq C. 
\end{align*}

\end{proof}

{\bf Two-dimensional parameter $\vartheta =\left(a,f\right)$}

Suppose that
 $$
\vartheta=\left(\theta_1,\theta_2\right)^\top
=\left(a,f\right)^\top\in \Theta= \left(\alpha _a,\beta _a\right)\times \left(\alpha
_f,\beta _f\right),  \qquad \alpha _a>0\qquad \alpha _f>0.
$$ The MME $\vartheta _\zT^*=   (\theta _{1,\zT}^*,\theta _{2,\zT}^*)^\top=   (a_\zT^*,f_\zT^*)^\top$ is defined as
solution $\vartheta =\vartheta _\zT^*$ of the system of equations
\begin{align*}
\Phi _1\left(\vartheta \right)=R_{1,T},\qquad \qquad \Phi _2\left(\vartheta
\right)=R_{2,T},\qquad \vartheta\in \Theta 
\end{align*}
with the  usual assumption that these equations have solution $\vartheta
_\zT^*\in\Theta  $. 

Recall that 
\begin{align*}
\Phi_1 \left(\vartheta \right)=\frac{f^2b^2}{ a^3}\left[e^{-a }-1+a
  \right]+\sigma^2,\qquad 
\quad \Phi_2 \left(\vartheta \right)=\frac{f^2b^2}{2a^3}\left[1-e^{-a }
  \right]^2.
\end{align*}
Therefore this system can be solved in two steps as follows:
the equation  
\begin{align}
\label{3-31}
\frac{e^{-a_\zT^*}-1+a_\zT^*}{\left(1-e^{-a_\zT^*}\right)^2}=\frac{R_{1,T}-\sigma
  ^2}{2R_{2,T}},
\end{align}
 can be solved the first  ($ $   and then having solution $a_\zT^* $ of this equation we
 define the second estimator
\begin{align}
\label{3-32}
 f_\zT^*=\frac{2(a_\zT^*)^{3/2}R_{2,T}^{1/2}}{b^2\left(1-e^{-a_\zT^*}\right)^2}.
\end{align}
Let us show that the function
\begin{align*}
h\left(x\right)=\frac{e^{-x}-1+x}{\left(1-e^{-x}\right)^2},\qquad x >0
\end{align*}
is strictly increasing.  Its derivative is $h'\left(x\right)
=\left(1-e^{-x}\right)^{-3}g\left(x\right)$ where the function
$g\left(x\right)=1-2xe^{-x}- e^{-2x}$ has the value $g\left(0\right)=0$ and
$g'\left(x\right)=2e^{-x}\left(x-1+ e^{-x}\right)>0 $ for $x>0$.  Therefore
the first equation has a unique solution $a_\zT^* $ and the second equation
gives an  explicit expression for the estimate $f_\zT^* $.

Note that
\begin{align*}
h\left(a_0\right)&=\frac{\Phi _1\left(\vartheta _0\right)-\sigma ^2}{2\Phi
  _2\left(\vartheta _0\right)} ,\qquad \frac{R_{1,T}-\sigma
  ^2}{2R_{2,T}}\longrightarrow h\left(a_0\right),\qquad a_\zT^*\longrightarrow a_0,\\
f_0&=\frac{2a_0^{3/2} \Phi
  _2\left(\vartheta _0\right)^{1/2}}{b^2\left(1-e^{-a_0}\right)},\qquad\qquad
\frac{2(a_\zT^*)^{3/2}R_{2,T}^{1/2}}{b^2\left(1-e^{-a_\zT^*}\right)^2}\longrightarrow 
f_0 
 . 
\end{align*}
Therefore
\begin{align*}
h\left(a_\zT^*\right)=h\left(a_0\right)+h'\left(a_0\right)\left(a_\zT^*-a_0
\right)+O\left( \left(a_\zT^*-a_0
\right)^2   \right)=\frac{R_{1,T}-\sigma
  ^2}{2R_{2,T}}
\end{align*}
and using the convergence \eqref{3-20} we obtain 
\begin{align*}
\sqrt{T}\left(a_\zT^*-a_0\right)&=h'\left(a_0\right)^{-1}\sqrt{T}\left(\frac{R_{1,T}-\sigma
  ^2}{2R_{2,T}}- \frac{\Phi _1\left(\vartheta _0\right)-\sigma ^2}{2\Phi
  _2\left(\vartheta _0\right)} \right)+O\left( \sqrt{T}\left(a_\zT^*-a_0
\right)^2   \right)\\
&=\frac{\sqrt{T}\left(R_{1,T}-\Phi _1\left(\vartheta
  _0\right)\right)}{2h'\left(a_0\right)\Phi _2\left(\vartheta
  _0\right)}-\frac{\left(\Phi _1\left(\vartheta
  _0\right)-\sigma ^2\right) }{2h'\left(a_0\right)\Phi _2\left(\vartheta
  _0\right)}\sqrt{T} \left(R_{2,T}-\Phi _2\left(\vartheta
_0\right)\right)+o\left(1\right)\\
&\Longrightarrow \frac{\xi _1}{2h'\left(a_0\right)\Phi _2\left(\vartheta
  _0\right)}-\frac{\left(\Phi _1\left(\vartheta
  _0\right)-\sigma ^2\right)\;\xi _2 }{2h'\left(a_0\right)\Phi _2\left(\vartheta
  _0\right)}\equiv \xi _3. 
\end{align*}
For the second estimator we have
\begin{align*}
&\sqrt{T}\left(f_\zT^*-f_0\right)=\sqrt{T}\left(\frac{2\left(a_\zT^*\right)^{3/2}R_{2,T}^{1/2}}{b^2
  \left(1-e^{-a_\zT^*}\right)^2}-\frac{2a_0^{3/2} \Phi _2\left(\vartheta
  _0\right)^{1/2}}{b^2\left(1-e^{-a_0}\right)} \right)\\
&\qquad = n'\left(a_0\right)\Phi _2\left(\vartheta
  _0\right)^{1/2} \sqrt{T}\left(a_\zT^*-a_0 \right)+   \frac{n\left(a_0\right)\sqrt{T}\left(R_{2,T} -\Phi _2\left(\vartheta
  _0\right)\right) }{2\Phi _2\left(\vartheta
  _0\right)^{1/2}}+o\left(1\right)\\
&\qquad \Longrightarrow  n'\left(a_0\right)\Phi _2\left(\vartheta
  _0\right)^{1/2}  \xi _3+   \frac{n\left(a_0\right) }{2\Phi _2\left(\vartheta
  _0\right)^{1/2}}\;\xi _2\equiv \xi _4.
\end{align*}
Let us denote $\xi^*=\left(\xi _3,\xi _4\right)^\top$ the Gaussian vector with
the corresponding covariance matrix ${\bf Q}\left(\vartheta _0\right)$.
Therefore we proved the following proposition.
\begin{proposition}
\label{3-P5} The MME $\vartheta _\zT^*$ is uniformly consistent, uniformly on
compacts $\KK \subset \Theta $ asymptotically normal
\begin{align*}
\sqrt{T}\left(\vartheta _\zT^*-\vartheta _0\right)\Longrightarrow \xi^*\sim
     {\cal N}\left(0,{\bf Q}\left(\vartheta _0\right)\right),
\end{align*}
the polynomial moments converge and the estimate
\begin{align}
\label{3-24}
\sup_{\vartheta _0\in\KK}T^{p/2}\Ex_{\vartheta _0}\left\|\vartheta _\zT^*-\vartheta _0\right\|^p\leq C
\end{align}
holds.
\end{proposition}
The proof of the estimate \eqref{3-24} can be carried out using standard
arguments.

{\bf Two-dimensional parameter $\vartheta =\left(a,b\right)$}

Recall that there is no real difference between estimation $f$ and $b$ because
the model depends on the product $fb$. Therefore for the definition of 
estimator $a_\zT^*$ see \eqref{3-31} and the estimator $b_\zT^*$   is 
  almost the same as   in \eqref{3-32}, where we replace $b$ by $f$
\begin{align*}
b_\zT^*=\frac{2(a_\zT^*)^{3/2}R_{2,T}^{1/2}}{f^2\left(1-e^{-a_\zT^*}\right)^2}.
\end{align*}

\section{One-step MLE-process.}

Consider once more the partially observed system
\begin{align*}
{\rm d}X_t&=f\left(\vartheta \right)Y_t\;{\rm d}t+\sigma\; {\rm d}W_t,\qquad
X_0=0,\qquad 0\leq t\leq T,\\
{\rm d}Y_t&=-a\left(\vartheta \right)Y_t\;{\rm d}t+b\left(\vartheta \right) {\rm
  d}V_t,\qquad  Y_0.
\end{align*}
The unknown parameter is one-dimensional, $\vartheta \in\Theta =\left(\alpha
,\beta \right)$ and we have to estimate it by observations
$X^T=\left(X_t,0\leq t\leq T\right)$. The initial value $Y_0\sim {\cal
  N}\left(0,d_y^2\right)$ does not depend of the independent Wiener processes
$W^T$ and $V^T$.   The MLE and BE described in Theorem \ref{T1} have nice
asymptotic properties and in particular are asymptotically efficient, but the
calculation of these estimators is computationally rather difficult problem.
For example, the construction of the BE \eqref{2-MleBe} requires the knowledge of
$L\left(\vartheta ,X^T\right)$  (see \eqref{2-LR-T}) for all $\vartheta \in\Theta $. Even if we
suppose that the system is in stationary regime with the initial value
$m\left(\vartheta ,0\right)\sim {\cal N}\left(0,\gamma_* \left(\vartheta
\right)\right)$, then we need not to seek the solution of Riccati equation,
but we need always the trajectories $m\left(\vartheta ,t\right), 0\leq t\leq
T,\vartheta \in\Theta $. 

The goal of this section is to propose and to study One-step MLE-process which
has many good properties, but first we will slightly simplify the exposition
supposing that the system is already in stationary regime, i.e. the initial
value $\gamma \left(\vartheta ,0\right)$ of Riccati equation is $\gamma
\left(\vartheta ,0\right)= \gamma _*\left(\vartheta \right)$ and therefore
$\gamma \left(\vartheta ,t\right)\equiv \gamma _*\left(\vartheta \right) $ for
all $t\geq 0$. Moreover, we put $M\left(\vartheta ,t\right)=f\left(\vartheta
\right)m(\vartheta ,t) $. Recall that $\Gamma \left(\vartheta \right)=\gamma_*
\left(\vartheta \right)f\left(\vartheta \right)^2\sigma ^{-2}=r\left(\vartheta
\right) -a\left(\vartheta \right) $, $r\left(\vartheta
\right)=\left(a\left(\vartheta \right)^2+ b(\vartheta
)^2f(\vartheta)^{2}\sigma ^{-2}\right)^{-1/2}$. Then the equations for
$M(\vartheta ,t),t\geq 0$ and $\dot M(\vartheta ,t),t\geq 0$ will be
\begin{align*}
{\rm d}M\left(\vartheta ,t\right)&=- r\left(\vartheta \right) M\left(\vartheta
,t\right){\rm d} t+ \Gamma 
\left(\vartheta \right) {\rm d}X_t,\\
{\rm d}\dot M\left(\vartheta ,t\right)&=- r\left(\vartheta \right)\dot
M\left(\vartheta ,t\right){\rm d} t- \dot r
  \left(\vartheta \right) M\left(\vartheta ,t\right){\rm d} t+  \dot  \Gamma
\left(\vartheta \right) {\rm d}X_t, 
\end{align*}
with the initial values $M\left(\vartheta ,0\right)$ and $\dot
M\left(\vartheta ,0\right)$.  The integral stationary presentation of
$M\left(\vartheta ,t\right),$ $t\geq 0$ is given in \eqref{2-13}. The integral
presentation for the derivative $\dot M\left(\vartheta ,t\right),t\geq 0$ can
be obtained by direct differentiation of the equation \eqref{2-13}. The
initial values $ M\left(\vartheta ,0\right),$ $\dot M\left(\vartheta ,0\right) $
can be obtained from these representations too.

The One-step MLE-process is constructed in two steps. First, by
observations $X^{\tau _\zT}$ on the relatively small {\it learning} interval
$\left[0,\tau _\zT\right],\tau _\zT=\left[T^\delta\right] , \delta
\in\left(\frac{1}{2},1\right) $ we construct a MME $\vartheta^* _{\tau
  _\zT}$ (see \eqref{3-21a}) and then using this estimator as preliminary one  we
propose One-step MLE-process. Here $\left[A\right]$ means the entire part of
$A$. 

Remind the notations
\begin{align*}
R_{1,\tau _\zT}&=\frac{1}{{\tau_\zT}}\sum_{k=1}^{\tau
  _\zT}\left[X_k-X_{k-1}\right]^2,\qquad \Psi \left(\vartheta \right)=
\frac{f(\vartheta )^2b(\vartheta )^2}{a(\vartheta )^2}\Bigl(e^{a(\vartheta
  )}-1+a(\vartheta )\Bigr)+\sigma ^2
,\\ 
\vartheta^*_{\tau_\zT}&={\rm arg}\inf_{\vartheta \in\Theta }\left|R_{1,\tau
  _\zT}-\Psi \left(\vartheta \right) \right| ,\qquad {\rm I}\left(\vartheta
\right)=\frac{\dot a\left(\vartheta \right)^2}{2a\left(\vartheta
  \right)}-\frac{2\dot a\left(\vartheta \right)\dot r\left(\vartheta
  \right)}{a\left(\vartheta \right)+r\left(\vartheta \right)}+\frac{\dot
  r\left(\vartheta \right)^2}{2r\left(\vartheta \right)}
\end{align*}
and the estimator-process
\begin{align}
\label{52}
\vartheta _{t,T}^\star=\vartheta^* _{\tau_\zT} +\frac{1}{{\rm
    I}\bigl(\vartheta^* _{\tau_\zT}
  \bigr)\left({t-{\tau_\zT}}\right)}\int_{{\tau_\zT}}^{t}\frac{\dot
  M(\vartheta^* _{\tau_\zT} ,s)}{\sigma ^2}\left[{\rm
    d}X_s-M(\vartheta^* _{\tau_\zT} ,s){\rm d}s\right],\;
          {\tau_\zT}<t\leq T.
\end{align}

The random processes  $M(\vartheta^* _{\tau_\zT} ,t),\dot M(\vartheta^*
_{\tau_\zT} ,t),\tau_\zT\leq t\leq T  $ satisfy the equations
\begin{align*}
{\rm d} M(\vartheta^* _{\tau_\zT} ,t)&=-r(\vartheta^*
_{\tau_\zT})M(\vartheta^* _{\tau_\zT} ,t){\rm d}t+ \Gamma (\vartheta^*
_{\tau_\zT} ){\rm d}X_t,\\
{\rm d}\dot  M(\vartheta^* _{\tau_\zT} ,t)&=-r(\vartheta^*
_{\tau_\zT})\dot M(\vartheta^* _{\tau_\zT} ,t){\rm d}t-\dot r(\vartheta^*
_{\tau_\zT}) M(\vartheta^* _{\tau_\zT} ,t){\rm d}t+ \dot \Gamma (\vartheta^*
_{\tau_\zT} ){\rm d}X_t.
\end{align*}
The initial values $M(\vartheta^* _{\tau_\zT} ,\tau_\zT),\dot M(\vartheta^*
  _{\tau_\zT} ,\tau_\zT) $ are defined as follows.
We have
\begin{align*}
M(\vartheta ,\tau_\zT)&= M\left(\vartheta,0\right)
e^{-r(\vartheta)\tau_\zT } +   \Gamma (\vartheta)e^{-r(\vartheta)\tau_\zT
}\int_{0}^{\tau_\zT}e^{r(\vartheta)s }{\rm d}X_s\\
&= M\left(\vartheta,0\right)
e^{-r(\vartheta)\tau_\zT } +   \Gamma (\vartheta)\left[X_{\tau_\zT }
  -r(\vartheta)\int_{0}^{\tau_\zT} e^{-r(\vartheta)\left(\tau_\zT-s \right)
}     X_s{\rm d}s\right],
\end{align*}
and the initial value is
\begin{align}
\label{53}
M(\vartheta^* _{\tau_\zT}  ,\tau_\zT)&= M(\vartheta^* _{\tau_\zT} ,0)
e^{-r(\vartheta^* _{\tau_\zT} )\tau_\zT }\nonumber\\
&\qquad \qquad  + \Gamma (\vartheta^* _{\tau_\zT} )\left[X_{\tau_\zT }
  -r(\vartheta^* _{\tau_\zT} )\int_{0}^{\tau_\zT} e^{-r(\vartheta^*
    _{\tau_\zT} )\left(\tau_\zT-s \right) 
}     X_s{\rm d}s\right].
\end{align}
For $\dot M(\vartheta^*
_{\tau_\zT} ,\tau_\zT)$ we have a similar representation
\begin{align}
\label{54}
\dot M(\vartheta^* _{\tau_\zT} ,\tau_\zT)&=\left[\dot M(\vartheta^*
  _{\tau_\zT} ,0)-M(\vartheta^* _{\tau_\zT} ,0) \dot r(\vartheta^* _{\tau_\zT}
  )\tau_\zT\right] e^{-r(\vartheta^* _{\tau_\zT} )\tau_\zT }+\dot \Gamma
(\vartheta^* _{\tau_\zT} )X_{\tau_\zT} \nonumber\\
 &\qquad -\left[ \dot \Gamma
  (\vartheta^* _{\tau_\zT} ) r (\vartheta^* _{\tau_\zT} )+\Gamma (\vartheta^*
  _{\tau_\zT} )\dot r (\vartheta^* _{\tau_\zT} )\right]\int_{0}^{\tau_\zT}
e^{-r(\vartheta^* _{\tau_\zT}\left(\tau_\zT-s \right) )} X_s{\rm
  d}s\nonumber\\
 &\qquad - \Gamma (\vartheta^* _{\tau_\zT} ) \dot
r(\vartheta^* _{\tau_\zT} )  r(\vartheta^* _{\tau_\zT}
)\int_{0}^{\tau_\zT} e^{-r(\vartheta^* _{\tau_\zT} )\left(\tau_\zT-s \right) }\left(\tau_\zT-s \right)
X_s{\rm d}s.
\end{align}

 Change the variables: $t=v T, v\in
 \left[\varepsilon_\zT,1\right],\varepsilon_\zT=\tau_\zT T^{-1}=T^{-1+\delta }\rightarrow 0 $
 and denote $\vartheta _\zT^\star\left(v\right)=\vartheta_{vT,T} ^\star,
 \varepsilon_\zT<v\leq 1 $.  
\begin{theorem}
\label{T2} Suppose that the following  conditions hold.
\begin{enumerate}
\item The functions $f\left(\vartheta
\right),a\left(\vartheta \right),b\left(\vartheta \right),\vartheta \in
\left[\alpha ,\beta \right]$ are   positive and twice 
continuously differentiable.
\item $\inf_{\vartheta \in\Theta }|\dot \Psi \left(\vartheta \right)|>0.   $

\item $ \inf_{\vartheta \in\Theta }\left(\left|\dot a\left(\vartheta
\right)\right|+\left|\dot r\left(\vartheta \right)\right|\right) >0   $.
\end{enumerate}
Then  One-step MLE-process  $\vartheta_\zT ^\star\left(v \right),
\varepsilon_\zT<v \leq 1  $ with $\delta \in \left(1/2, 1\right)$ is uniformly
on compacts $\KK\subset\Theta $
consistent: for any $\nu >0$ and any $\varepsilon _0 \in (0,1]$
\begin{align*}
\lim_{T\rightarrow \infty }\sup_{\vartheta _0\in\KK}\Pb_{\vartheta
  _0}\left(\sup_{\varepsilon _0\leq v\leq 1}\left|\vartheta_\zT 
^\star\left(v \right)-\vartheta _0 
\right|>\nu \right)=  0, 
\end{align*}
 asymptotically normal
\begin{align*}
\sqrt{T}\left(\vartheta_\zT ^\star\left(v \right)-\vartheta _0
\right)\Longrightarrow \zeta _v\sim {\cal N}\left(0,v^{-1} {\rm
  I}\left(\vartheta_0 \right)^{-   1}\right),
\end{align*}
and the moments   converge: for any $p\geq 1$ 
\begin{align*}
T^{p/2}\Ex_{\vartheta _0}\left|\vartheta_\zT ^\star\left(v \right)-\vartheta
_0 \right|^p\longrightarrow  \Ex_{\vartheta _0}\left|\zeta _v \right|^p.
\end{align*}
The random process  $\eta _\zT\left(v\right)=v\sqrt{T{\rm
  I}\left(\vartheta_0 \right)}\left(\vartheta_\zT ^\star\left(v \right)-\vartheta _0
\right),\varepsilon _0\leq v\leq 1 $
converges in distribution in the measurable space (${\cal
  C}\left[\varepsilon _0,1\right],  {\goth B}$) to the Wiener process
$W\left(v\right),\varepsilon _0\leq v\leq 1 $.
\end{theorem}
\begin{proof}  
Let us denote ${\rm I}_m=\inf_{\vartheta \in\Theta }{\rm I}\left(\vartheta
\right)>0$, suppose that $\varepsilon _T<\varepsilon _0$ and consider the
probability
\begin{align*}
& \Pb_{\vartheta _0}\left(\sup_{\varepsilon _0\leq v\leq 1}\left|\vartheta_\zT ^\star\left(v
  \right)-\vartheta_0 \right|>\nu \right)\leq \Pb_{\vartheta
    _0}\left(|\vartheta^*_{\tau _\zT} -\vartheta_0 |>\frac{\nu
  }{3}\right)\\ 
&\quad +\Pb_{\vartheta _0}\left(\frac{1}{{\rm I}_m
    T\left(\varepsilon _0-\varepsilon _\zT\right)} \sup_{\varepsilon _0\leq v\leq 1}\left| \int_{\varepsilon
    _\zT}^{v}\frac{\dot M(\vartheta^*_{\tau _\zT},uT ) }{\sigma
  } {\rm d}\bar W_{uT}\right|\geq \frac{\nu }{3} \right)\\ 
&\quad  +\Pb_{\vartheta _0}\left(\frac{1}{{\rm I}_m \left(\varepsilon _0-\varepsilon
    _\zT\right)} \int_{\varepsilon _\zT}^{1}\frac{\left| \dot
    M(\vartheta^*_{\tau _\zT},uT )\left[M(\vartheta_0,uT
      )-M(\vartheta^*_{\tau _\zT},uT ) \right] \right|}{\sigma^2
  } {\rm d}u\geq \frac{\nu }{3} \right).
\end{align*}
 For the first probability due to \eqref{3-21} we have
\begin{align*}
\Pb_{\vartheta
    _0}\left(\left|\vartheta^*_{\tau _\zT} -\vartheta_0 \right|>\frac{\nu
  }{3}\right)\leq \frac{9}{\nu ^2}\Ex_{\vartheta
  _0}|\vartheta^*_{\tau _\zT} -\vartheta_0 |^2\leq  \frac{C}{\nu^2 \tau _\zT}\rightarrow 0.
\end{align*}
The second probability is estimated as follows 
\begin{align*}
&\Pb_{\vartheta _0}\left( \sup_{\varepsilon _0\leq
  v\leq 1}\left| \int_{\varepsilon 
    _\zT}^{v}\frac{\dot M(\vartheta^*_{\tau _\zT},uT ) }{\sigma
  } {\rm d}\bar W_{uT}\right|\geq \frac{\nu {\rm I}_m
    T\left(\varepsilon _0-\varepsilon _\zT\right) }{3} \right)\\
&\qquad \qquad \qquad \qquad \qquad  \leq \frac{C}{\nu^2 {\rm I}_m^2 
    T\left(\varepsilon _0-\varepsilon _\zT\right)^2 } \int_{\varepsilon 
    _\zT}^{1}\frac{\Ex_{\vartheta
    _0}\dot M(\vartheta^*_{\tau _\zT},uT )^2 }{\sigma^2
  } {\rm d}u  \\
&\qquad \qquad \qquad \qquad \qquad   \leq \frac{C}{\nu^2 {\rm I}_m^2 
    T\left(\varepsilon _0-\varepsilon _\zT\right)^2 } \longrightarrow 0.
\end{align*}
Further
\begin{align*}
&\Ex_{\vartheta _0} \left|\dot
    M(\vartheta^*_{\tau _\zT},uT )\left[M\left(\vartheta_0,uT
      \right)-M(\vartheta^*_{\tau _\zT},uT ) \right]\right|\\
&\qquad \qquad =\Ex_{\vartheta _0} \left|\dot
    M(\vartheta^*_{\tau _\zT},uT )\dot
    M\left(\bar\vartheta_{\tau _\zT},uT \right)\left[\vartheta^*_{\tau
        _\zT}-\vartheta _0\right]\right|\\
&\qquad \qquad \leq  \left(\Ex_{\vartheta _0} \left|\dot
    M(\vartheta^*_{\tau _\zT},uT )\dot
    M\left(\bar\vartheta_{\tau _\zT},uT \right)\right|^2 \Ex_{\vartheta _0} [\vartheta^*_{\tau
        _\zT}-\vartheta _0]^2\right)^{1/2}\\
&\qquad \qquad \leq C \left( \Ex_{\vartheta _0} [\vartheta^*_{\tau
        _\zT}-\vartheta _0]^2\right)^{1/2}\leq  CT^{-\delta/2 }\rightarrow 0.
\end{align*}
Therefore the last probability tends to zero as well. All estimates are
uniform on compacts. 

The estimator  $\vartheta_\zT ^\star\left(v
  \right) $ has the representation
\begin{align*}
&\sqrt{T}\left(\vartheta_\zT ^\star\left(v \right)-\vartheta
_0\right)=\sqrt{T}\left(\vartheta_{\tau _\zT} ^\star\left(v
\right)-\vartheta _0\right) +\frac{\sqrt{T}}{{\rm
    I}\bigl(\vartheta^* _{\tau_\zT}
  \bigr)\left({t-{\tau_\zT}}\right)}\int_{{\tau_\zT}}^{t}\frac{\dot
  M(\vartheta^* _{\tau_\zT} ,s)}{\sigma }{\rm
    d}\bar W_s\\
&\qquad \qquad +\frac{\sqrt{T}}{{\rm
    I}\bigl(\vartheta^* _{\tau_\zT}
  \bigr)\left({t-{\tau_\zT}}\right)}\int_{{\tau_\zT}}^{t}\frac{\dot
  M(\vartheta^* _{\tau_\zT} ,s)}{\sigma ^2}\left[M( \vartheta _{0}
  ,s)-M(\vartheta^* _{\tau_\zT} ,s)\right]{\rm d}s\\
 &\qquad  =\frac{1}{{\rm
    I}\bigl(\vartheta^* _{\tau_\zT}
  \bigr)\left({v-\varepsilon _\zT}\right)}  \frac{1}{\sqrt{T}}  \int_{{\tau_\zT}}^{t}\frac{\dot
  M(\vartheta^* _{\tau_\zT} ,s)}{\sigma }{\rm
    d}\bar W_s\\
&\qquad \qquad  +\frac{\sqrt{T}\left(\vartheta^* _{\tau_\zT} -\vartheta _0
  \right)}{{\rm
    I}\bigl(\vartheta^* _{\tau_\zT}
  \bigr)   }\left({\rm
    I}\bigl(\vartheta^* _{\tau_\zT}
  \bigr)-\frac{1}{\left({t-{\tau_\zT}}\right)}\int_{{\tau_\zT}}^{t}\frac{\dot
  M(\vartheta^* _{\tau_\zT} ,s)\dot
  M(\bar\vartheta _{\tau_\zT} ,s)}{\sigma ^2}{\rm d}s\right).
\end{align*}
The stochastic integral has the following asymptotics
\begin{align*}
&\Ex_{\vartheta _0}\left( \frac{1}{\sqrt{T}}  \int_{{\tau_\zT}}^{t}\frac{\dot
  M(\vartheta^* _{\tau_\zT} ,s)}{\sigma }{\rm
    d}\bar W_s-\frac{1}{\sqrt{T}}  \int_{{\tau_\zT}}^{t}\frac{\dot
  M(\vartheta _0 ,s)}{\sigma }{\rm
    d}\bar W_s  \right)^2\\
&\qquad =\frac{1}{{T}}  \int_{{\tau_\zT}}^{t}\frac{\Ex_{\vartheta _0}\left[\dot
  M(\vartheta^* _{\tau_\zT} ,s)  -\dot  M(\vartheta _0 ,s)  \right] ^2}{\sigma ^2 }{\rm
    d}s \leq C\frac{t}{T}\Ex_{\vartheta _0}\left[\vartheta^* _{\tau_\zT}
    -\vartheta _0  \right]^2\leq  C\frac{t}{T}T^{-\delta }\rightarrow 0
\end{align*}
and for a fixed $ v$
\begin{align*}
\frac{1}{\sqrt{T}} \int_{{\tau_\zT}}^{vT}\frac{\dot M(\vartheta _0 ,s)}{\sigma
}{\rm d}\bar W_s\Longrightarrow {\cal N}\left(0,v {\rm I}\left(\vartheta
_0\right)\right)
\end{align*}
because by the law of large numbers 
\begin{align*}
\frac{1}{{T}} \int_{{\tau_\zT}}^{vT}\frac{\dot M(\vartheta _0 ,s)^2}{\sigma^2
}{\rm d}s\longrightarrow  v{\rm I}\left(\vartheta
_0\right).
\end{align*}
Therefore
\begin{align}
\label{56}
\frac{1}{{\rm I}\left(\vartheta^* _{\tau_\zT} \right)\left(v-\varepsilon  _\zT\right)\sqrt{T} }
\int_{{\tau_\zT}}^{t}\frac{\dot M(\vartheta^* _{\tau_\zT} ,s)}{\sigma
}{\rm d}\bar W_s\Longrightarrow  {\cal N}\left(0,v^{-1} {\rm I}\left(\vartheta
_0\right)^{-1}\right).
\end{align}
For the last term we write
\begin{align*}
&\int_{{\tau_\zT}}^{t}\frac{\dot M(\vartheta^* _{\tau_\zT} ,s)\dot
  M(\bar\vartheta _{\tau_\zT} ,s)}{\sigma ^2}{\rm d}s\\
&\qquad \quad =\int_{{\tau_\zT}}^{t}\frac{\dot
  M(\vartheta^* _{\tau_\zT} ,s)^2}{\sigma ^2}{\rm d}s+ \left(\bar\vartheta
  _{\tau_\zT}-\vartheta^* _{\tau_\zT}       \right)\int_{{\tau_\zT}}^{t}\frac{\dot 
  M(\vartheta^* _{\tau_\zT} ,s)\ddot
  M(\bar{\bar\vartheta} _{\tau_\zT} ,s)   }{\sigma ^2}{\rm d}s,
\end{align*}
where  $ \left|\bar\vartheta _{\tau_\zT} -\vartheta^* _{\tau_\zT}   \right|\leq
|\vartheta^* _{\tau_\zT}-\vartheta _0 |  $. 
 Hence
\begin{align*}
&\frac{\sqrt{T}\left(\vartheta^* _{\tau_\zT} -\vartheta _0
  \right)}{{\rm
    I}\bigl(\vartheta^* _{\tau_\zT}
  \bigr)   }\left({\rm
    I}\bigl(\vartheta^* _{\tau_\zT}
  \bigr)-\frac{1}{\left({t-{\tau_\zT}}\right)}\int_{{\tau_\zT}}^{t}\frac{\dot
  M(\vartheta^* _{\tau_\zT} ,s)\dot
  M(\bar\vartheta _{\tau_\zT} ,s)}{\sigma ^2}{\rm d}s\right)\\
&\qquad \qquad =\frac{\sqrt{T}\left(\vartheta^* _{\tau_\zT} -\vartheta _0
  \right)}{{\rm
    I}\bigl(\vartheta^* _{\tau_\zT}
  \bigr)   }\left({\rm
    I}\bigl(\vartheta^* _{\tau_\zT}
  \bigr)-\frac{1}{\left({t-{\tau_\zT}}\right)}\int_{{\tau_\zT}}^{t}\frac{\dot
  M(\vartheta^* _{\tau_\zT} ,s)^2}{\sigma ^2}{\rm d}s\right)\\
&\qquad \qquad\qquad +\frac{\sqrt{T}\left(\vartheta^* _{\tau_\zT} -\vartheta _0
  \right)\left(\bar\vartheta _{\tau_\zT} -\vartheta _0 \right)}{{\rm
    I}\bigl(\vartheta^* _{\tau_\zT}
  \bigr) \left(t-\tau _\zT\right)  }\int_{{\tau_\zT}}^{t}\frac{\dot 
  M(\vartheta^* _{\tau_\zT} ,s)\ddot
  M(\bar{\bar\vartheta} _{\tau_\zT} ,s)   }{\sigma ^2}{\rm d}s.
\end{align*}
 Note that
\begin{align*}
\sqrt{T}\left|\left(\vartheta^* _{\tau_\zT} -\vartheta _0
  \right)\left(\bar\vartheta _{\tau_\zT} -\vartheta _0 \right)\right|\leq
  \sqrt{T}\left|\vartheta^* _{\tau_\zT} -\vartheta _0
  \right|^2=O\left(T^{1/2 - \delta }\right)\longrightarrow 0.
\end{align*}
The similar estimates allows us to write
\begin{align*}
&{\rm I}\bigl(\vartheta^* _{\tau_\zT}
  \bigr)-\frac{1}{\left({t-{\tau_\zT}}\right)}\int_{{\tau_\zT}}^{t}\frac{\dot
    M(\vartheta^* _{\tau_\zT} ,s)^2}{\sigma ^2}{\rm d}s
  ={\rm I}\bigl(\vartheta _{0}
  \bigr)-\frac{1}{\left({t-{\tau_\zT}}\right)}\int_{{\tau_\zT}}^{t}\frac{\dot
    M(\vartheta_0,s)^2}{ \sigma ^2}{\rm d}s+O\left(\vartheta^* _{\tau_\zT}-\vartheta
  _0 \right).
\end{align*}
Further 
\begin{align*}
\Ex_{\vartheta _0}\left( \frac{1}{\left({t-{\tau_\zT}}\right)}\int_{{\tau_\zT}}^{t}\left[\frac{\dot
  M(\vartheta _{0} ,s)^2}{\sigma ^2}  -  {\rm
    I}\bigl(\vartheta _{0}
  \bigr)  \right] {\rm d}s \right)^2\leq \frac{C}{t}
\end{align*}
and we obtain 
\begin{align*}
\frac{\sqrt{T}\left(\vartheta^* _{\tau_\zT} -\vartheta _0
  \right)}{{\rm
    I}\bigl(\vartheta^* _{\tau_\zT}
  \bigr)   }\left({\rm
    I}\bigl(\vartheta^* _{\tau_\zT}
  \bigr)-\frac{1}{\left({t-{\tau_\zT}}\right)}\int_{{\tau_\zT}}^{t}\frac{\dot
  M(\vartheta^* _{\tau_\zT} ,s)\dot
  M(\bar\vartheta _{\tau_\zT} ,s)}{\sigma ^2}{\rm d}s\right)\longrightarrow 0.
\end{align*}
This  convergence and \eqref{56} yield the asymptotic normality.

The convergence of moments  follows from  \eqref{3-21} and boundedness of all
moments of the derivatives of $M\left(\vartheta ,s\right)$.

It is proved that the process $\eta _\zT\left(v\right), \varepsilon _0,\leq
v\leq 1$ admits the representation
\begin{align}
\label{58a}
\eta _\zT\left(v\right)&=\frac{1}{\sigma \sqrt{{\rm I}\left(\vartheta
    _0\right)T   }}\int_{0}^{vT}\dot M\left(\vartheta _0,s\right){\rm d}\bar
W_s+h_\zT\left(v\right) ,
\end{align}
where the residue $h_\zT(v)=h_\zT(v,\vartheta
_0,\vartheta^* _{\tau_\zT} )\rightarrow 0 $ uniformly on $\vartheta
_0\in\KK$
\begin{align*}
\sup_{\vartheta _0\in\KK}\Ex_{\vartheta _0}h_\zT(v,\vartheta
_0,\vartheta^* _{\tau_\zT} )^2\rightarrow 0
\end{align*}
 and such that
\begin{align}
\label{58}
\Ex_{\vartheta _0}\left|\eta _\zT\left(v_2\right)-\eta _\zT\left(v_1\right)
\right|^4\leq C\left|v_2-v_1\right|^2 + C\left|v_2-v_1\right|^4 \leq
C\left|v_2-v_1\right|^2.
\end{align}
To verify the convergence of two-dimensional distributions $ \eta
_\zT\left(v_1\right),\eta _\zT\left(v_2\right) $ we introduce $\langle \lambda
,\eta _\zT\rangle= \lambda _1\eta _\zT\left(v_1\right)+\lambda _2\eta
_\zT\left(v_2\right) $ and note that 
\begin{align*}
\langle \lambda ,\eta _\zT\rangle =\frac{1}{\sigma \sqrt{{\rm
      I}\left(\vartheta _0\right)T }}\int_{0}^{T}\left[\lambda
  _1\1_{\left\{s<v_1T\right\}}+\lambda _2\1_{\left\{s<v_2T\right\}}\right]\dot
M\left(\vartheta _0,s\right){\rm d}\bar W_s+  \langle \lambda
, h_\zT\rangle
\end{align*}
where $\langle \lambda , h_\zT\rangle\rightarrow 0 $ and
\begin{align*}
&\frac{1}{\sigma ^2{{\rm I}\left(\vartheta _0\right)T
  }}\int_{0}^{T}\left[\lambda _1\1_{\left\{s<v_1T\right\}}+\lambda
    _2\1_{\left\{s<v_2T\right\}}\right]^2\dot M\left(\vartheta
  _0,s\right)^2{\rm d}s\\ &\qquad =\frac{\lambda _1^2}{\sigma ^2{{\rm
        I}\left(\vartheta _0\right)T }}\int_{0}^{v_1T}\dot M\left(\vartheta
  _0,s\right)^2{\rm d}s+\frac{\lambda _2^2}{\sigma ^2{{\rm I}\left(\vartheta
      _0\right)T }}\int_{0}^{v_2T}\dot M\left(\vartheta _0,s\right)^2{\rm
    d}s\\ &\qquad \qquad +\frac{2\lambda _1\lambda _2}{\sigma ^2{{\rm I}\left(\vartheta
      _0\right)T }}\int_{0}^{\left(v_1\wedge v_2\right)T}\dot M\left(\vartheta
  _0,s\right)^2{\rm d}s\\
&\qquad \longrightarrow \lambda _1^2v_1+2\lambda _1\lambda _2\left(v_1\wedge v_2\right)+\lambda _2^2v_2
\end{align*}
because
\begin{align*}
\frac{v}{\sigma ^2{{\rm I}\left(\vartheta
      _0\right)Tv }}\int_{0}^{vT}\dot M\left(\vartheta
  _0,s\right)^2{\rm d}s\longrightarrow v.
\end{align*}
Therefore by the CLT
\begin{align*}
\langle \lambda ,\eta _\zT\rangle \Longrightarrow \langle \lambda ,W\rangle
=\lambda _1W\left(v_1\right)+\lambda _2W\left(v_2\right) .
\end{align*}

The convergence  of finite-dimensional distributions 
 $$
\left(\eta
_\zT\left(v_1\right),\ldots, \eta _\zT\left(v_k\right)
\right)\Longrightarrow\left(W\left(v_1\right),\ldots, W\left(v_k\right)
\right)
 $$
 can be verified by a similar way. This convergence and \eqref{58} provide
the weak convergence $\eta _\zT\left(\cdot \right)\Longrightarrow W\left(\cdot
\right) $.

\end{proof}

\begin{remark}
\label{R1}
{\rm The representation \eqref{52} of the One-step MLE-process can be given
  in the   recurrent form as follows. 
For the values  $t\in (\tau_\zT, T] $  we have the equality
\begin{align*}
 \left({t-{\tau_\zT}}\right)\vartheta _{t,T}^\star=\left({t-{\tau_\zT}}\right)\vartheta^* _{\tau_\zT} +\frac{1}{{\rm
    I}\bigl(\vartheta^* _{\tau_\zT}
  \bigr)}\int_{{\tau_\zT}}^{t}\frac{\dot
  M(\vartheta^* _{\tau_\zT} ,s)}{\sigma ^2}\left[{\rm
    d}X_s-M(\vartheta^* _{\tau_\zT} ,s){\rm d}s\right].          
\end{align*}
Hence
\begin{align}
\label{59}
{\rm d}\vartheta _{t,T}^\star=\frac{\vartheta^* _{\tau_\zT}-\vartheta _{t,T}^\star }{{t-{\tau_\zT}}}{\rm d}t+\frac{{\dot
  M(\vartheta^* _{\tau_\zT} ,t)}{}}{{\rm
    I}\bigl(\vartheta^* _{\tau_\zT}
  \bigr)\sigma ^2\left({t-{\tau_\zT}}\right)}\left[{\rm
    d}X_t-M(\vartheta^* _{\tau_\zT} ,t){\rm d}t\right].
\end{align}
To avoid the singularities like $0/0$ at the beginning $t=\tau_\zT$ in
numerical simulations the value $t-{\tau_\zT}$ in this formula can be replaced
by $t-{\tau_\zT}+\varepsilon ^*$, where $\varepsilon ^*>0 $ is some small
value. This modification, of course, will not change the asymptotic results of
$\vartheta _{t,T}^\star $.

}
\end{remark}

{\bf Estimation of the parameter $f$.}

If $f\left(\vartheta \right)=\vartheta \in\Theta =\left(\alpha _f,\beta
_f\right), \alpha _f>0,$ $a\left(\vartheta \right)=a>0$ and
$b\left(\vartheta \right)=b>0$, then the preliminary estimator is
\begin{align*}
\vartheta^* _{\tau_\zT}=\left(\frac{a^3}{b^2\left[e^{-a}-1+a\right]{\tau_\zT}}
\sum_{k=1}^{{\tau_\zT}}\left(\left[X_{k}-X_{k-1}\right]^2-\sigma^2\right)
\right)^{1/2}.
\end{align*}
The processes $M\left(\vartheta ,t\right),\dot M\left(\vartheta
,t\right),t\geq 0$ are solutions of the equations
\begin{align*}
{\rm d}M\left(\vartheta ,t\right)&=-r\left(\vartheta \right)M\left(\vartheta
,t\right){\rm d}t+\Gamma \left(\vartheta \right){\rm d}X_t,\qquad\qquad t\geq 0,\\
{\rm d}\dot M\left(\vartheta ,t\right)&=-r\left(\vartheta \right)\dot M\left(\vartheta
,t\right){\rm d}t+\frac{b^2\vartheta }{r \left(\vartheta \right)\sigma
  ^2}\;{\rm d}X_t     ,
\end{align*}
where $r\left(\vartheta \right)=\left(a^2+b^2\vartheta ^2\sigma
^{-2}\right)^{1/2}$ and $\Gamma \left(\vartheta \right)=r\left(\vartheta
\right)-a $. 

The One-step MLE-process is
\begin{align*}
\vartheta _{t,T}^\star=\vartheta^* _{\tau_\zT}+\frac{2\sigma ^3r(\vartheta^* _{\tau_\zT}
  )^3}{b^4\vartheta^{*2} _{\tau_\zT}\left(t-\tau _\zT \right)} \int_{\tau _\zT}^{t}\dot
M(\vartheta^* _{\tau_\zT},s)\left[{\rm d}X_s-M(\vartheta^*
  _{\tau_\zT},s){\rm d}s\right],\quad  \tau _\zT<t\leq T .
\end{align*}

According to Theorem \ref{T2} this estimator has the properties mentioned
in 
this theorem and in particular we have the asymptotic normality
\begin{align*}
\sqrt{T}\left(\vartheta _{\tau_\zT}^\star\left(v\right)-\vartheta _0\right)\Longrightarrow {\cal
N}\left(0,\frac{2\sigma ^4r\left(\vartheta_0
  \right)^3}{v\, b^4\vartheta _0^2} \right).
\end{align*}

\subsubsection{Estimation of the parameter $b$.}

If $b\left(\vartheta \right)=\vartheta \in\Theta =\left(\alpha _b,\beta
_b\right), \alpha _b>0,$ $f\left(\vartheta \right)=f>0$ and
$a\left(\vartheta \right)=a>0$, then the preliminary estimator and One-step
MLE-process are 
\begin{align*}
\vartheta^* _{\tau _\zT}&=\left(\frac{a^3}{f^2\left[e^{-a}-1+a\right]\tau _\zT}
\sum_{k=1}^{\tau _\zT}\left(\left[X_{k}-X_{k-1}\right]^2-\sigma^2\right)
\right)^{1/2},\\
\vartheta _{t,T}^\star&=\vartheta^* _{\tau _\zT}+\frac{2\sigma ^3r(\vartheta^* _{\tau _\zT}
  )^3}{f^4\,\vartheta^{*2} _{\tau _\zT}\left(t-\tau _\zT \right)} \int_{\tau _\zT}^{t}\dot
M(\vartheta^* _{\tau _\zT},s)\left[{\rm d}X_s-M(\vartheta^*
  _{\tau _\zT},s){\rm d}s\right],\quad  \tau _\zT<t\leq T 
\end{align*}
with the  corresponding  processes $M\left(\vartheta ,t\right),t\geq 0$ and $\dot M\left(\vartheta
,t\right),t\geq 0$. 

We have as well 
\begin{align*}
\sqrt{T}\left(\vartheta _\zT^\star\left(v\right)-\vartheta _0\right)\Longrightarrow {\cal
N}\left(0,\frac{2\sigma ^4r\left(\vartheta_0
  \right)^3}{v\, f^4\vartheta _0^2} \right).
\end{align*}

\subsubsection{Estimation of the parameter $a$.}

If $a\left(\vartheta \right)=\vartheta \in\Theta =\left(\alpha _a,\beta
_a\right), \alpha _a>0,$ $f\left(\vartheta \right)=f>0$ and
$a\left(\vartheta \right)=a>0$, then the preliminary estimator is
\begin{align*}
\vartheta^* _{\tau _\zT}={\rm arg}\inf_{\vartheta \in\Theta
}\left|\frac{1}{\tau _\zT}\sum_{k=1}^{\tau _\zT}\left[X_{k}-X_{k-1}\right]^2-\Psi\left(\vartheta
\right) \right|,\qquad \Psi\left(\vartheta \right)=\frac{f^2b^2}{\vartheta
  ^3}\left[e^{-\vartheta }-1+\vartheta \right]+\sigma ^2.
\end{align*}
 The Fisher information ${\rm I}\left(\vartheta \right)$  is defined in
\eqref{2-43}
\begin{align*}
{\rm I}\left(\vartheta_0 \right)=\frac{ \Gamma \left(\vartheta
  _0\right)\left[r\left(\vartheta _0\right)+\vartheta _0+r\left(\vartheta
    _0\right)\vartheta _0\Gamma \left(\vartheta
  _0\right) \right]}{2\vartheta _0r\left(\vartheta
  _0\right)^3\left[r\left(\vartheta _0\right)+\vartheta _0\right]},
\end{align*}
and 
 the One-step MLE-process is given by \eqref{52}. It  is
asymptotically normal
\begin{align*}
\sqrt{T}\left(\vartheta _\zT^\star\left(v\right)-\vartheta
_0\right)\Longrightarrow {\cal N}\left(0,v^{-1}{\rm I}\left(\vartheta_0
\right)^{-1} \right).
\end{align*}

\subsubsection{Estimation of the two-dimensional parameter $\left(a,f\right)$.}

The partially observed system is 
\begin{align*}
{\rm d}X_t&=\theta _2Y_t{\rm d}t+\sigma {\rm d}W_t,\qquad
X_0,\qquad  0\leq t\leq T,\\
{\rm d}Y_t&=-\theta _1Y_t{\rm d}t+b{\rm
  d}V_t,\qquad  Y_0,\qquad t\geq 0,
\end{align*}
where $b\not=0$ is known and we have to estimate $\vartheta
=\left(a,f\right)=\left(\theta _1,\theta _2\right)\in\Theta =\left(\alpha
_a,\beta _a\right)\times \left(\alpha _f,\beta _f\right) $, $\alpha _a>0,\alpha _f>0$. 

Recall that the preliminary MME $\vartheta^*_{\tau _\zT}=(\theta^*_{1,\tau
  _\zT},\theta^*_{2,\tau _\zT})=(a^*_{\tau _\zT},f^*_{\tau _\zT}) $
can be calculated as follows
\begin{align*}
&\frac{e^{-a_{\tau _\zT}^*}-1+a_{\tau _\zT}^*}{(1-e^{-a_{\tau
        _\zT}^*})^2}=\frac{R_{1,\tau _\zT}-\sigma ^2}{2R_{2,\tau
      _\zT}},\qquad \quad R_{1,\tau _\zT}=\frac{1}{\tau _\zT}\sum_{k=1}^{\tau
    _\zT}\left[X_{k}-X_{k-1}\right]^2,\\ & f_{\tau _\zT}^*=\frac{2(a_{\tau 
      _\zT}^*)^{3/2}\,R_{2,\tau _\zT}^{1/2}}{b^2(1-e^{-a_{\tau
        _\zT}^*})^2},\qquad\qquad  R_{2,\tau _\zT}=\frac{1}{\tau _\zT}\sum_{k=2  
}^{\tau
    _\zT}\left[X_{k}-X_{k-1}\right]\left[X_{k-1}-X_{k-2}\right].
\end{align*}

Fisher information matrix is
\begin{align}
 {\bf I}\left(\vartheta_0 \right)=\left(\begin{array}{cc}
   \frac{\Gamma\left(\vartheta _0\right)^2 \left[r\left(\vartheta_0
       \right)\theta _{1,0} +\left(r\left(\vartheta_0
       \right)+\theta _{1,0}\right)^2
       \right]}{2\theta _{1,0} r\left(\vartheta_0
     \right)^3\left[r\left(\vartheta_0 \right)+\theta _{1,0} \right] },
   &\frac{b^2\theta _{2,0}\left(\theta _{1,0}^2+r\left(\vartheta _0\right)\theta
     _{1,0}-2r\left(\vartheta _0\right)\right) }{ 2r\left(\vartheta
     _0\right)^3\left(r\left(\vartheta _0\right)+\theta _{1,0} \right)\sigma ^2 }
   \\\frac{b^2\theta _{2,0} \left(\theta _{1,0}^2+r\left(\vartheta _0\right)\theta
     _{1,0}-2r\left(\vartheta _0\right) \right)}{ 2r\left(\vartheta
     _0\right)^3\left(r\left(\vartheta _0\right)+\theta _{1,0} \right)\sigma ^2 },
   &\frac{\theta _{2,0} ^2b^4}{2r\left(\vartheta_0 \right)^3\sigma ^4}
 \end{array}
\right).
\label{51}
\end{align} 
 and it is non-degenerate. For its calculation we used the stationary
 representations for the derivatives $\dot M_f\left(\vartheta _0,t\right)$ and
 $\dot M_a\left(\vartheta _0,t\right)$.  One-step MLE-process is the vector
 process $\vartheta _{t,T}^\star,\tau _\zT<t\leq T$ defined by the expression
\begin{align*}
\vartheta _{t,T}^\star=\vartheta^*_{\tau _\zT}+\frac{{\bf I}(\vartheta^*_{\tau
    _\zT})^{-1}}{\sigma \left(t-\tau _\zT \right)}\int_{\tau _\zT}^{t}\dot
    M(\vartheta^*_{\tau _\zT},s) \left[{\rm d}X_s-  M(\vartheta^*_{\tau
        _\zT},s){\rm d}s\right] .
\end{align*}
Here $M(\vartheta^*_{\tau
        _\zT},s),\tau _\zT\leq t\leq T $ and vector $\dot M(\vartheta^*_{\tau
        _\zT},s)=(\dot   M_a(\vartheta^*_{\tau _\zT},t) ,\dot   M_f(\vartheta^*_{\tau _\zT},t))^\top ,\tau _\zT\leq t\leq T $    are solutions of the equations
\begin{align*}
{\rm d}M(\vartheta^*_{\tau _\zT},t)&=-r(\vartheta^*_{\tau
  _\zT})M(\vartheta^*_{\tau _\zT},t){\rm d}t +\Gamma (\vartheta^*_{\tau
  _\zT}){\rm d}X_t, \\
{\rm d}\dot   M_a(\vartheta^*_{\tau _\zT},t)&=-r(\vartheta^*_{\tau
  _\zT})\dot M_a(\vartheta^*_{\tau _\zT},t){\rm d}t -\dot r_a(\vartheta^*_{\tau
  _\zT})M(\vartheta^*_{\tau _\zT},t){\rm d}t +\dot \Gamma_a (\vartheta^*_{\tau
  _\zT}){\rm d}X_t,\\
{\rm d}\dot   M_f(\vartheta^*_{\tau _\zT},t)&=-r(\vartheta^*_{\tau
  _\zT})\dot M_f(\vartheta^*_{\tau _\zT},t){\rm d}t -\dot r_f(\vartheta^*_{\tau
  _\zT})M(\vartheta^*_{\tau _\zT},t){\rm d}t +\dot \Gamma_f (\vartheta^*_{\tau
  _\zT}){\rm d}X_t,
\end{align*}
The initial values $ M(\vartheta^*_{\tau _\zT},\tau _\zT), \dot
M_a(\vartheta^*_{\tau _\zT},\tau _\zT)$  and $ \dot  M_f(\vartheta^*_{\tau _\zT},\tau _\zT) $     are calculated with the help of the representations
\eqref{53},\eqref{54} given  above. 

Of course this expression for One-step MLE-process we can replace by the
recurrent version like \eqref{59}
\begin{align}
\label{}
{\rm d}\vartheta _{t,T}^\star=\frac{\vartheta^* _{\tau_\zT}-\vartheta _{t,T}^\star }{{t-{\tau_\zT}}}{\rm d}t+\frac{{\bf
    I}\bigl(\vartheta^* _{\tau_\zT}
  \bigr)^{-1}{\dot
  M(\vartheta^* _{\tau_\zT} ,t)}{}}{\sigma ^2\left({t-{\tau_\zT}}\right)}\left[{\rm
    d}X_t-M(\vartheta^* _{\tau_\zT} ,t){\rm d}t\right].
\end{align}

\begin{proposition}
\label{P6} One-step MLE-process $\vartheta
_{t,T}^\star,\tau _\zT<t\leq T$ is uniformly consistent, asymptotically normal
\begin{align*}
\sqrt{vT}\left(\vartheta _{vT,T}^\star-\vartheta _0  \right)\Longrightarrow
\zeta \sim {\cal N}\left(0,{\bf I}(\vartheta_0)^{-1}\right)
\end{align*}
and the polynomial moments converge.
\end{proposition}
The proof follows the same lines as the proof of Theorem \ref{T2}.

 \section{Adaptive filter}

Consider once more the same partially observed linear system
\begin{align*}
{\rm d}X_t&=f\left(\vartheta \right)Y_t{\rm d}t+\sigma {\rm d}W_t,\qquad
X_0,\qquad 0\leq t\leq T,\\
{\rm d}Y_t&=-a\left(\vartheta \right)Y_t{\rm d}t+b\left(\vartheta \right){\rm
  d}V_t,\qquad  Y_0,\qquad t\geq 0,
\end{align*}
where $W_t,t\geq 0$ and $V_t,t\geq 0$ are independent Wiener processes. Here
the functions $f\left(\vartheta \right),a\left(\vartheta
\right),b\left(\vartheta \right),\vartheta \in\Theta  $ are known, positive
and smooth. The value of the parameter $\vartheta \in\Theta=\left(\alpha
,\beta \right)$ is unknown.  

The conditional expectation $m\left(\vartheta ,t\right)=\Ex_\vartheta
\left(Y_t|X_s,0\leq s\leq t\right)$ satisfies the Kalman-Bucy equations
\begin{align}
{\rm d}m\left(\vartheta ,t\right)&=-\left[a\left(\vartheta
  \right)+\frac{\gamma \left(\vartheta ,t\right)f\left(\vartheta
    \right)^2}{\sigma ^2}\right]m\left(\vartheta ,t\right) {\rm
  d}t+\frac{\gamma \left(\vartheta ,t\right)f\left(\vartheta \right)}{\sigma
  ^2}{\rm d}X_t,\quad m\left(\vartheta ,0\right),\nonumber\\ \frac{\partial
  \gamma \left(\vartheta ,t\right)}{\partial t}&=-2a\left(\vartheta \right)
\gamma \left(\vartheta ,t\right)-\frac{\gamma \left(\vartheta
  ,t\right)^2f\left(\vartheta \right)^2}{\sigma ^2}+b\left(\vartheta \right)^2
,\qquad \qquad \gamma \left(\vartheta ,0\right),
\label{62}
\end{align}
but to use these equations for the calculation of $m\left(\vartheta
,t\right),t\geq 0$ is impossible because the parameter $\vartheta $ is
unknown. Below we consider the problem of approximation $m\left(\vartheta
,t\right), 0\leq t\leq T$ by observations $X^T=\left(X_t,0\leq t\leq
T\right)$. The main idea is to use the One-step MLE-process on the time
interval $t\in (\tau _\zT, T]$
\begin{align*}
{\rm d}\vartheta _{t,T}^\star=\frac{\vartheta^* _{\tau_\zT}-\vartheta
  _{t,T}^\star }{{t-{\tau_\zT}}}{\rm d}t+\frac{{\rm I}\bigl(\vartheta^*
  _{\tau_\zT} \bigr)^{-1}{\dot M(\vartheta^* _{\tau_\zT} ,t)}{}}{\sigma
  ^2\left({t-{\tau_\zT}}\right)}\left[{\rm d}X_t-M(\vartheta^* _{\tau_\zT}
  ,t){\rm d}t\right].
\end{align*}
where $\vartheta^* _{\tau _\zT} $ is a MME constructed by observations
$X^{\tau _\zT}=\left(X_s,0\leq s\leq \tau _\zT\right)$, $\tau _\zT=T^\delta ,
\delta \in (1/2,1)$, $M\left(\vartheta ,s\right)=f\left(\vartheta \right)
m\left(\vartheta ,s\right) $ and to replace $\vartheta $ by   $
\vartheta_{t,T}^\star$ in the equations of 
filtration.

There are at least three possibilities of such  approximation of the random
process $m\left(\vartheta ,t\right),0\leq t\leq T$.

Consider the random process $\hat m_{t,T},\tau _\zT<t\leq T$ obtained as
solution of the equation
\begin{align*}
{\rm d}\hat m_{t,T}&=-a\left(\vartheta
_{t,T}^\star \right)\hat m_{t,T}{\rm d}t+\frac{\hat \gamma_{t,T}f\left(\vartheta _{t,T}^\star\right)}{\sigma ^2}\left[{\rm d}X_t-f\left(\vartheta
_{t,T}^\star\right)\hat m_{t,T}{\rm d}t\right],\quad \hat m_{\tau _\zT,T},
\end{align*}
where
\begin{align*}
\frac{\partial \hat\gamma_{t,T}}{\partial t}&=-2a\left(\vartheta _{t,T}^\star
\right)\hat \gamma_{t,T}-\frac{\hat \gamma _{t,T}^2f\left(\vartheta
  _{t,T}^\star \right)^2}{\sigma ^2}+b\left(\vartheta _{t,T}^\star \right)^2
,\qquad \qquad \hat \gamma_{ \tau _\zT  ,T}=\gamma \left(\vartheta^* _{\tau
  _\zT},\tau _\zT  \right). 
\end{align*}
Of course, a special attention has to be paid for the initial value $\hat
m_{\tau _\zT,T}$ (see \eqref{53}).

Another possibility is to use the explicit expression for the solution of the
equation \eqref{62}
\begin{align*}
\gamma \left(\vartheta ,t\right)=e^{-2r\left(\vartheta
  \right)t}\left[\frac{1}{\gamma\left(\vartheta ,0\right)-\gamma _*\left(\vartheta
    \right)}+\frac{f\left(\vartheta
    \right)^2}{2r\left(\vartheta
    \right)\sigma ^2}\left(1-e^{-2r\left(\vartheta
    \right)t}\right)\right]^{-1}+\gamma _*\left(\vartheta \right),
\end{align*}
where
\begin{align*}
r\left(\vartheta \right)=\left(a\left(\vartheta
    \right)^2+\frac{f\left(\vartheta
    \right)^2b\left(\vartheta
    \right)^2}{\sigma
  ^2}\right)^{1/2},\qquad\quad  \gamma _*\left(\vartheta \right)=
\frac{\sigma
  ^2\left[r\left(\vartheta
    \right)-a\left(\vartheta
    \right)\right]}{f\left(\vartheta
    \right)^2}.
\end{align*}
The approximation of $m\left(\vartheta ,t\right)$ can be realized by the
random process $\hat m_\zT\left(t\right),\tau _\zT<t\leq T$
\begin{align*}
{\rm d}\hat m_\zT\left(t\right)&=-a\left(\vartheta
_{t,T}^\star \right)\hat m_\zT\left(t\right){\rm d}t+\frac{ \gamma\left(
  \vartheta _{t,T}^\star,t\right)  f\left(\vartheta _{t,T}^\star\right)}{\sigma
  ^2}\left[{\rm d}X_t-f\left(\vartheta 
_{t,T}^\star\right)\hat m_\zT\left(t\right){\rm d}t\right],
\end{align*}
with the initial value  $m_\zT\left(\tau _\zT  \right) $.

The third opportunity is  to use the exponential convergence of the solution
$\gamma \left(\vartheta ,t\right)$ of the Riccati  to the steady
state value  $\gamma _*\left(\vartheta \right)$ and the approximation of
$m\left(\vartheta ,t\right)$ can be done with the help of the solution
$m_{t,T}^\star ,\tau _\zT<t\leq T$ of the equation
\begin{align}
\label{68}
{\rm d}m_{t,T}^\star=-r\left(\vartheta_{t,T}^\star \right)m_{t,T}^\star{\rm
  d}t+ B\left(\vartheta_{t,T}^\star \right)\; {\rm
  d}X_t,\qquad m_{\tau _\zT ,T}^\star
\end{align}
where we denoted  $B\left(\vartheta \right)=\left[r\left(\vartheta
  \right)-a\left(\vartheta \right)\right]f\left(\vartheta \right)^{-1}$.
To define $m_{\tau _\zT ,T}^\star $ we  can use the equality like
\eqref{52}, i.e., we put
\begin{align*}
m_{\tau _\zT ,T}^\star&=m(\vartheta^*_{\tau _\zT},0 )e^{-r(\vartheta^*_{\tau
    _\zT} )\tau _\zT }+B(\vartheta^*_{\tau _\zT} )\left[X_{\tau
    _\zT}-r(\vartheta^*_{\tau _\zT} ) \int_{0}^{\tau
    _\zT}e^{-r(\vartheta^*_{\tau _\zT} )\left(\tau _\zT-s \right) }X_s\,{\rm d}s   \right].
\end{align*}
Introduce the constants
\begin{align*}
K_1\left(\vartheta _0\right)&=-\frac{\left[\dot a\left(\vartheta _0\right)+\dot
    f\left(\vartheta _0\right)B
  \left(\vartheta _0\right)\right]\sigma }{f\left(\vartheta
  _0\right)\sqrt{2a(\vartheta _0){\rm I}\left(\vartheta _0\right)} },\\
K_2\left(\vartheta _0\right)&=\frac{\dot r\left(\vartheta _0\right)\sigma }{f\left(\vartheta
  _0\right)\sqrt{2r (\vartheta _0){\rm I}\left(\vartheta _0\right)} },\qquad \quad 
R_{1,2}\left(\vartheta
_0\right)= \frac{2K_1\left(\vartheta _0\right)K_2\left(\vartheta _0\right)\sqrt{a\left(\vartheta _0\right)r\left(\vartheta _0\right) }
}{a\left(\vartheta _0\right)+r\left(\vartheta _0\right)} ,
\end{align*}
and random processes
\begin{align*}
\xi _\zT^{\left(1\right)}\left(v\right)&=\sqrt{2a\left(\vartheta _0\right)T}\int_{v-\nu
  _\zT}^{v}e^{-a\left(\vartheta _0\right)\left(v-u\right)T}{\rm d}W_T\left(u
\right)\Longrightarrow \xi ^{\left(1\right)}\left(v\right)\sim {\cal
  N}\left(0,1\right)    ,\quad v\in \left[\varepsilon _0,1\right],\\
\xi _\zT^{\left(2\right)}\left(v\right)&=\sqrt{2r \left(\vartheta _0\right)T}\int_{v-\nu
  _\zT}^{v}e^{-r \left(\vartheta _0\right)\left(v-u\right)T}{\rm d}W_T\left(u
\right)\Longrightarrow \xi ^{\left(2\right)}\left(v\right)\sim {\cal
  N}\left(0,1\right)    ,\quad v\in \left[\varepsilon _0,1\right].
\end{align*}
Here $\varepsilon _0\in (0,1)$, $\nu _\zT=T^{-\kappa },\kappa \in \left(0,1\right)$ and note  that 
\begin{align*}
\Ex_{\vartheta _0}\xi ^{(1)}\left(v\right)\xi
^{(2)}\left(v\right)&=\frac{2\sqrt{a\left(\vartheta _0\right)r
    \left(\vartheta _0\right)}}{a\left(\vartheta _0\right)+r
  \left(\vartheta _0\right)},\quad \qquad \Ex_{\vartheta _0}\xi ^{(1)}\left(v_1\right)\xi
^{(2)}\left(v_2\right)=0
\end{align*}
if  $v_1\not=v_2$  because
\begin{align*}
\Ex_{\vartheta _0}\xi^{(1)} _\zT\left(v_1\right)\xi^{(2)}
_\zT\left(v_2\right)=\sqrt{r\left(\vartheta _0\right)a \left(\vartheta
  _0\right)}T \int_{\left(v_1-\nu _\zT\right)\vee \left(v_2-\nu
  _\zT\right)}^{v_1\wedge v_2}e^{-\left[r\left(\vartheta _0\right)+a
    \left(\vartheta _0\right)\right]\left(v-u\right)T}\1_{\AA_1\cap \AA_2}{\rm d}u=0
\end{align*}
for  $\left|v_2-v_1\right|>\nu _\zT\rightarrow 0$ and the sets $\AA_i=\left\{u:v_i-\nu
_\zT\leq u\leq  v_i\right\} $. 

Denote  the $2\times 2$ matrix $ \rho \left(\vartheta _0\right)$ with the
components $\rho _{i,j}\left(\vartheta _0\right)=\Ex_{\vartheta 
  _0}\xi^{(j)}\left(v\right)\xi^{(j)}\left(v\right)$. The Gaussian vector $\xi\left(v\right)
=\left(\xi^{(1)}\left(v\right) ,\xi^{(2)}\left(v\right) \right) \sim {\cal N}\left(0,\rho
\left(\vartheta _0\right)\right)$. 

 Recall that the
random process $\left(\eta _\zT\left(v\right), v\in
\left[\varepsilon _0,1\right]\right) $ converges in distribution to the Wiener
process $\left(W\left(v\right), v\in \left[\varepsilon _0,1\right]\right)
$ for any  $\varepsilon _0\in \left(0,1\right)$ (Theorem \ref{T2}).
The direct calculations allow to verify that 
\begin{align*}
\Ex_{\vartheta _0}\xi^{(i)} _\zT\left(v\right)\eta
_\zT\left(v\right)\longrightarrow \Ex_{\vartheta _0}\xi^{(i)}
\left(v\right)W\left(v\right)=0. 
\end{align*}
Therefore the vector $\xi\left(v\right) $ and $W\left(v\right)$ are
independent and the vectors $\xi^{(1)}  \left(v_1\right) $, $\xi^{(2)}\left(v_2\right) $
for $v_1\not=v_2$ are independent too.

\begin{theorem}
\label{T5} Suppose that the conditions of  Theorem \ref{T2}
hold. Then   for  $ v\in 
\left[\varepsilon _0,1\right]$   we have the convergence in distribution
\begin{align}
\label{69}
\sqrt{T}\left(m_{vT,T}^\star-m\left(\vartheta
_0,vT\right)\right)&\Longrightarrow
\frac{W\left(v\right)}{v}\left[K_1\left(\vartheta _0\right) \xi
^{\left(1\right)}\left(v\right)+K_2\left(\vartheta _0\right) \xi
^{\left(2\right)}\left(v\right)\right]
\end{align}
and the convergence of moments
\begin{align}
\label{70}
T\,\Ex_{\vartheta _0}\left(m_{vT,T}^\star-m\left(\vartheta
_0,vT\right)\right)^2\longrightarrow \frac{1}{v}\left[K_1\left(\vartheta
_0\right)^2+K_2\left(\vartheta _0\right)^2+2     R _{1,2}\left(\vartheta _0\right)\right].
\end{align}
\end{theorem}
\begin{proof} Let us denote
\begin{align*}
r_m=\inf_{\vartheta \in\Theta } r\left(\vartheta \right),\qquad \quad
B_M=\sup_{\vartheta \in\Theta } B\left(\vartheta \right) . 
\end{align*}
We need two estimates. To obtain the first  one we write 
\begin{align*}
m_{t,T}^\star&=m_{\tau _\zT ,T}^\star e^{-\int_{\tau _\zT}^{t}r(\vartheta
  _{s,T}^\star){\rm d}s}+ f\left(\vartheta _0\right)\int_{\tau
  _\zT}^{t}e^{-\int_{\tau _\zT}^{t}r(\vartheta 
  _{s,T}^\star){\rm d}s}B\left(\vartheta
  _{s,T}^\star \right)m\left(\vartheta
_0,s\right){\rm d}s \\
&\qquad +\sigma \int_{\tau _\zT}^{t}e^{-\int_{\tau _\zT}^{t}r(\vartheta
  _{s,T}^\star){\rm d}s}B\left(\vartheta
  _{s,T}^\star \right) {\rm d}\bar W_s
\end{align*}
and the estimate is 
\begin{align*}
\Ex_{\vartheta _0} \left|m_{t,T}^\star \right|^2&\leq 3e^{-r_m t}
\Ex_{\vartheta _0} \left|m_{\tau _\zT ,T}^\star \right|^2 +3\sigma
^2B_M^2\int_{\tau _\zT}^{t}e^{-  r_m \left(t-s\right)}{\rm d}s \\
&\; +3f\left(\vartheta
_0\right)^2B_M^2  \int_{\tau _\zT}^{t}e^{-  r_m\left(t-s\right)} \int_{\tau
  _\zT}^{t}e^{-  r_m\left(t-q\right)} \Ex_{\vartheta _0} m\left(\vartheta
_0,s\right)m\left(\vartheta _0,q\right){\rm d}s{\rm d}q \leq C ,
\end{align*}
where the constant $C>0$ does not depend on $t$. 

Similarly we can verify the estimate: for any $p\geq 2$ there exists a constant
$C=C\left(p\right)>0$ such that
\begin{align*}
\sup_{\vartheta _0\in\KK}\sup_{\tau _\zT\leq  t\leq T}\Ex_{\vartheta _0} \left|m_{t,T}^\star \right|^p&\leq C.
\end{align*}
To obtain the second estimate we write for the difference $\Delta
_{t,T}=m_{t,T}^\star-m\left(\vartheta _0,t\right) $ and $\Delta
_{vT,T}$ the representations 
\begin{align*}
\Delta _{t,T}& = \int_{\tau
  _\zT}^{t}e^{-r\left(\vartheta
  _0\right)\left(t-s\right)}\left[r\left(\vartheta _0\right)-r\left(\vartheta
  _{s,T}^\star\right)\right]m_{s,T}^\star{\rm d}s \\ 
&\qquad \quad +\int_{\tau
  _\zT}^{t}e^{-r\left(\vartheta
  _0\right)\left(t-s\right)}\left[B\left(\vartheta
  _{s,T}^\star\right)-B\left(\vartheta _0\right)\right]f\left(\vartheta
_0\right)m\left(\vartheta _0,s\right){\rm d}s\\ 
&\qquad \quad +\sigma \int_{\tau _\zT}^{t}e^{-r\left(\vartheta
  _0\right)\left(t-s\right)}\left[B\left(\vartheta
  _{s,T}^\star\right)-B\left(\vartheta _0\right)\right]{\rm d}\bar W_s\\
& = T\int_{\varepsilon 
  _\zT}^{v}e^{-r\left(\vartheta
  _0\right)\left(v-u\right)T}\left[r\left(\vartheta _0\right)-r\left(\vartheta
  _{uT,T}^\star\right)\right]m_{uT,T}^\star{\rm d}u \\ 
&\qquad \quad +T\int_{\varepsilon 
  _\zT}^{v}e^{-r\left(\vartheta
  _0\right)\left(v-u\right)T}\left[B\left(\vartheta
  _{uT,T}^\star\right)-B\left(\vartheta _0\right)\right]f\left(\vartheta
_0\right)m\left(\vartheta _0,uT\right){\rm d}u\\ 
&\qquad \quad +\sqrt{T}\sigma \int_{\varepsilon  _\zT}^{v}e^{-r\left(\vartheta
  _0\right)\left(v-u\right)T}\left[B\left(\vartheta
  _{uT,T}^\star\right)-B\left(\vartheta _0\right)\right]{\rm d}
W_T\left(u\right)\equiv \Delta _{vT,T},
\end{align*}
where $\varepsilon _\zT=T^{\delta -1}$  and  $W_T\left(u\right)=T^{-1/2}\bar
W_{uT} $ is a Wiener process.  

 Let us put $\nu
_\zT=T^{-\kappa }, \kappa \in (0,1)$. Then we have the estimate
\begin{align*}
&T^{3/2}\int_{\varepsilon _\zT}^{v-\nu 
  _\zT}e^{-r\left(\vartheta
  _0\right)\left(v-u\right)T}\left|\left[r\left(\vartheta _0\right)-r\left(\vartheta
  _{uT,T}^\star\right)\right]m_{uT,T}^\star\right|{\rm d}u\\
&\qquad \qquad \qquad \leq CT^{3/2}e^{-r\left(\vartheta
  _0\right)T^{1-\kappa }}\int_{\varepsilon _\zT}^{v-\nu 
  _\zT}\left|m_{uT,T}^\star\right|{\rm d}u\rightarrow 0.
\end{align*}
We study $\Delta _{vT,T} $ for the values $v\in (\varepsilon _0, 1)$, where
$\varepsilon _0\in (0,1)$.  Therefore it is sufficient to study the following
integrals
\begin{align*}
\sqrt{T}\Delta _{vT,T}& = T^{3/2}\int_{v-\nu_\zT}^{v}e^{-r\left(\vartheta
  _0\right)\left(v-u\right)T}\left[r\left(\vartheta _0\right)-r\left(\vartheta
  _{uT,T}^\star\right)\right]m_{uT,T}^\star{\rm d}u \\ 
&\qquad \quad +T^{3/2}\int_{v-\nu_\zT}^{v}e^{-r\left(\vartheta
  _0\right)\left(v-u\right)T}\left[B\left(\vartheta
  _{uT,T}^\star\right)-B\left(\vartheta _0\right)\right]f\left(\vartheta
_0\right)m\left(\vartheta _0,uT\right){\rm d}u\\ 
&\qquad \quad +\sigma{T}\int _{v-\nu_\zT}^{v}e^{-r\left(\vartheta
  _0\right)\left(v-u\right)T}\left[B\left(\vartheta
  _{uT,T}^\star\right)-B\left(\vartheta _0\right)\right]{\rm d}\bar W_T\left(u\right)+o\left(1\right).
\end{align*}

We have
\begin{align*}
\Ex_{\vartheta _0}\left(\Delta _{t,T}\right)^2& \leq 3 \Ex_{\vartheta _0}\left(\int_{\tau
  _\zT}^{t}e^{-r\left(\vartheta
  _0\right)\left(t-s\right)}\left[r\left(\vartheta _0\right)-r\left(\vartheta
  _{s,T}^\star\right)\right]m_{s,T}^\star{\rm d}s\right)^2 \\ 
&\qquad \quad +3\Ex_{\vartheta _0}\left(\int_{\tau
  _\zT}^{t}e^{-r\left(\vartheta
  _0\right)\left(t-s\right)}\left[B\left(\vartheta
  _{s,T}^\star\right)-B\left(\vartheta _0\right)\right]f\left(\vartheta
_0\right)m\left(\vartheta _0,s\right){\rm d}s\right)^2\\ 
&\qquad \quad +3\sigma^2 \int_{\tau _\zT}^{t}e^{-2r\left(\vartheta
  _0\right)\left(t-s\right)}\Ex_{\vartheta _0} \left[B\left(\vartheta
  _{s,T}^\star\right)-B\left(\vartheta _0\right)\right]^2{\rm d}s.
\end{align*}
Recall that the moments of One-step MLE-process  and  the derivatives  $\dot r\left(\vartheta \right)$ and $\dot
B\left(\vartheta \right)$ are bounded. Hence 
\begin{align*}
&\Ex_{\vartheta _0}\left| r\left(\vartheta _0\right)-r\left(\vartheta
  _{uT,T}^\star \right) \right|^2\leq C\Ex_{\vartheta _0}\left| \vartheta _0 - \vartheta
  _{uT,T}^\star  \right|^2\leq CT^{-1 },\\
&\Ex_{\vartheta _0}\left| B\left(\vartheta _{uT,T}^\star\right)-B\left(\vartheta
  _0 \right) \right|^2\leq C\Ex_{\vartheta _0}\left|  \vartheta
  _{uT,T}^\star-\vartheta _0   \right|^2\leq CT^{-1}
\end{align*}
and using elementary estimates we obtain 
\begin{align*}
\Ex_{\vartheta _0}\left(\Delta _{t,T}\right)^2=o\left(1\right)\qquad {\rm and}
\qquad m_{t,T}^\star=m\left(\vartheta _0,t\right)+O\left(\delta _T\right),
\end{align*}
where we denoted $\delta _T=\vartheta
  _{uT,T}^\star-\vartheta _0 $.

We are interested by the asymptotic ($T\rightarrow \infty $) behavior of the
estimation errors   $\Delta _{t,T} $ and  $T\Ex_{\vartheta _0}\Delta _{t,T}^2
$. That is why we  consider the 
re-parametrization $t=vT$ and study $\Delta _{vT,T} $.      Note that
\begin{align*}
f\left(\vartheta _0\right)\left[B(\vartheta _{uT,T}^\star)-B\left(\vartheta
  _0\right)\right]&=\frac{f\left(\vartheta _0\right)\Gamma (\vartheta
  _{uT,T}^\star)}{f (\vartheta _{uT,T}^\star)}-\Gamma (\vartheta_0)\\
 &=\Gamma (\vartheta _{uT,T}^\star)-\Gamma \left(\vartheta
_0\right)+\frac{(f\left(\vartheta _0\right)-f (\vartheta _{uT,T}^\star)) }{f
  (\vartheta _{uT,T}^\star)} \Gamma (\vartheta _{uT,T}^\star)
\end{align*}
and
\begin{align*}
&\left[r\left(\vartheta _0\right)-r\left(\vartheta
    _{uT,T}^\star\right)\right]m_{uT,T}^\star+f\left(\vartheta
  _0\right)\left[B(\vartheta _{uT,T}^\star)-B\left(\vartheta
    _0\right)\right]m\left(\vartheta _0,uT\right)\\
&\qquad =\left[r\left(\vartheta _0\right)-r\left(\vartheta
    _{uT,T}^\star\right)\right]m\left(\vartheta
  _0,uT\right)\left(1+O\left(\delta _T\right)\right)+\left[\Gamma (\vartheta
  _{uT,T}^\star)-\Gamma \left(\vartheta _0\right)\right]m\left(\vartheta
  _0,uT\right)\\
&\qquad \qquad \qquad +\frac{(f\left(\vartheta _0\right)-f (\vartheta
  _{uT,T}^\star))  }{f (\vartheta _{uT,T}^\star)} \Gamma (\vartheta 
  _{uT,T}^\star)m\left(\vartheta  _0,uT\right)\\
&\qquad =-\left[\dot a\left(\vartheta _0\right)+ \frac{\dot f\left(\vartheta
      _0\right)}{f\left(\vartheta _0\right)}\Gamma 
  \left(\vartheta _0\right) \right]   m\left(\vartheta  _0,uT\right)\delta
  _T+O\left(\delta _T^2\right).
\end{align*}
These relations allow us to write
\begin{align*}
& \sqrt{T}\Delta _{vT,T}= T^{3/2}\int_{v-\nu 
    _\zT}^{v}e^{-r\left(\vartheta
    _0\right)\left(v-u\right)T}\left[r\left(\vartheta
    _0\right)-r\left(\vartheta _{uT,T}^\star\right)\right]m\left(\vartheta
  _0,uT\right) {\rm d}u\left(1+o\left(\delta _T\right)\right) \\ 
& \qquad +T^{3/2}\int_{v-\nu 
    _\zT}^{v}e^{-r\left(\vartheta
    _0\right)\left(v-u\right)T}\left[\frac{f\left(\vartheta _0\right)\Gamma \left(\vartheta
    _{uT,T}^\star\right)}{f(\vartheta
    _{uT,T}^\star)}-\Gamma \left(\vartheta _0\right)\right]m\left(\vartheta
  _0,uT\right){\rm d}u\left(1+o\left(\delta _T\right)\right)\\ 
&\qquad +\sigma{T}\int _{v-\nu 
    _\zT}^{v}e^{-r\left(\vartheta
    _0\right)\left(v-u\right)T}\left[B\left(\vartheta
    _{uT,T}^\star\right)-B\left(\vartheta _0\right)\right]{\rm d}\bar
  W_T\left(u\right)\\
&\quad= -\left[\dot a\left(\vartheta _0\right)+ \frac{\dot f\left(\vartheta
      _0\right)}{f\left(\vartheta _0\right)}\Gamma 
  \left(\vartheta _0\right) \right]  T^{3/2}\int_{v-\nu 
    _\zT}^{v}e^{-r\left(\vartheta
    _0\right)\left(v-u\right)T}\left(\vartheta _{uT,T}^\star-\vartheta
  _0\right)    m\left(\vartheta   _0,uT\right){\rm d}u\\
&\qquad +\dot B\left(\vartheta _0\right)\sigma T\int_{v-\nu 
    _\zT}^{v}e^{-r\left(\vartheta
    _0\right)\left(v-u\right)T}\left(\vartheta _{uT,T}^\star-\vartheta
  _0\right)    {\rm d}\bar W_{T}\left(u\right)+o\left(1\right)\\
&\quad=C\left(\vartheta_0 \right)  T\int_{v-\nu 
    _\zT}^{v}e^{-r\left(\vartheta
    _0\right)\left(v-u\right)T}u^{-1}\eta _\zT\left(u\right)    m\left(\vartheta
  _0,uT\right){\rm d}u\\ 
&\qquad +D\left(\vartheta_0 \right) \sqrt{T}\int_{v-\nu 
    _\zT}^{v}e^{-r\left(\vartheta 
    _0\right)\left(v-u\right)T} u^{-1}  \eta _\zT\left(u\right)   {\rm d}\bar
  W_{T}\left(u\right)+o\left(1\right),
\end{align*}
where 
\begin{align*}
C\left(\vartheta_0 \right)= -\left[\dot a\left(\vartheta _0\right)+ \frac{\dot
    f\left(\vartheta _0\right)}{f\left(\vartheta _0\right)}\Gamma
  \left(\vartheta _0\right) \right]{\rm I}\left(\vartheta _0\right)^{-1/2}
,\qquad D\left(\vartheta_0 \right)=\dot B\left(\vartheta _0\right){\rm
  I}\left(\vartheta _0\right)^{-1/2}\sigma.
\end{align*}
We have the estimate
\begin{align*}
&\Ex_{\vartheta _0}T\int_{v-\nu _\zT}^{v}e^{-r\left(\vartheta
    _0\right)\left(v-u\right)T}\left|u^{-1}\eta _\zT\left(u\right)- v^{-1}\eta
  _\zT\left(v\right) \right| \left| m\left(\vartheta _0,uT\right)\right|{\rm
    d}u \\
 &\qquad \leq T\int_{v-\nu _\zT}^{v}e^{-r\left(\vartheta
    _0\right)\left(v-u\right)T}\left(\Ex_{\vartheta _0}\left|u^{-1}\eta
  _\zT\left(u\right)- v^{-1}\eta _\zT\left(v\right) \right|^2 \Ex_{\vartheta
    _0}\left| m\left(\vartheta _0,uT\right)\right|^2\right)^{1/2}{\rm d}u\\
 &\qquad \leq C T\int_{v-\nu _\zT}^{v}e^{-r\left(\vartheta
    _0\right)\left(v-u\right)T}\left|v-u\right|^{1/2}{\rm d}u\leq C\nu
  _T^{1/2}\longrightarrow 0, 
\end{align*}
where we used the estimate \eqref{58}.

Therefore
\begin{align*}
& T\int_{v-\nu 
    _\zT}^{v}e^{-r\left(\vartheta 
    _0\right)\left(v-u\right)T} u^{-1}  \eta _\zT\left(u\right)
m\left(\vartheta _0,uT\right)  {\rm d}u\nonumber\\ 
&\qquad \qquad = v^{-1}  \eta _\zT\left(v\right)  T\int_{v-\nu 
    _\zT}^{v}e^{-r\left(\vartheta 
    _0\right)\left(v-u\right)T}  m\left(\vartheta _0,uT\right)  {\rm d}u+O( \nu
  _T^{1/2}   ).
\end{align*}
The substitution of 
\begin{align*}
m\left(\vartheta _0,uT\right) =\frac{\Gamma\left(\vartheta _0\right)\sigma }
 { f\left(\vartheta _0\right)}\sqrt{T} \int_{-\infty }^{u}e^{-a\left(\vartheta
   _0\right)\left(u-z\right)T }{\rm d}\bar W_T\left(z\right)
\end{align*}
in the last integral gives us the following expression
\begin{align*}
&T\int_{v-\nu _\zT}^{v}e^{-r\left(\vartheta _0\right)\left(v-u\right)T}
  m\left(\vartheta _0,uT\right) {\rm d}u\\ &\qquad \quad
  =\frac{\Gamma\left(\vartheta _0\right)\sigma } { f\left(\vartheta
    _0\right)}T^{3/2}\int_{v-\nu _\zT}^{v}e^{-r\left(\vartheta
    _0\right)\left(v-u\right)T} \int_{-\infty }^{u}e^{-a\left(\vartheta
    _0\right)\left(u-z\right)T }{\rm d}\bar W_T\left(z\right)\,{\rm
    d}u\\ &\qquad \quad =\frac{\Gamma\left(\vartheta _0\right)\sigma } {
    f\left(\vartheta _0\right)}T^{3/2}e^{-r\left(\vartheta
    _0\right)vT}\int_{-\infty }^{v}e^{a\left(\vartheta
    _0\right)zT }\int_{\left(v-\nu _\zT\right)\vee z }^{v}e^{\Gamma \left(\vartheta
    _0\right)uT}{\rm d}u\;{\rm d}\bar W_T\left(z\right)\\
 &\qquad \quad =\frac{\sigma\sqrt{T} } {f\left(\vartheta
    _0\right)}e^{-r\left(\vartheta 
    _0\right)vT}\int_{-\infty }^{v}e^{a\left(\vartheta
    _0\right)zT }\left[e^{\Gamma \left(\vartheta
    _0\right)vT}-e^{\Gamma \left(\vartheta
    _0\right) \left[\left(v-\nu _\zT\right)\vee z \right]  T}\right]\;{\rm
    d}\bar W_T\left(z\right)\\
 &\qquad \quad =\frac{\sigma\sqrt{T} } {f\left(\vartheta
    _0\right)}\int_{v-\nu _T}^{v}\left[e^{-a\left(\vartheta
    _0\right)\left(v-z\right)T }- e^{-r\left(\vartheta
    _0\right)\left(v-z\right)T }    \right]\;{\rm
    d}\bar W_T\left(z\right)\left(1+o\left(1\right)\right).
\end{align*}
Remark that
\begin{align*}
&\Ex_{\vartheta _0}\left(\sqrt{T} \int_{-\infty  }^{v-\nu _T}e^{-r
  \left(\vartheta _0\right)\left(v-z\right)T }\;{\rm d}\bar W_T\left(z\right)
\right)^2\\
&\qquad \qquad  =T\int_{-\infty  }^{v-\nu _T}e^{-2r
  \left(\vartheta _0\right)\left(v-z\right)T }\;{\rm d}z=\frac{1}{2r
  \left(\vartheta _0\right) }e^{-2r
  \left(\vartheta _0\right) \nu _TT }\longrightarrow 0.
\end{align*}
Therefore we obtained the representation
\begin{align*}
&T\int_{v-\nu _\zT}^{v}e^{-r\left(\vartheta _0\right)\left(v-u\right)T}
  m\left(\vartheta _0,uT\right) {\rm d}u\\ &\qquad \quad
  =\frac{\sigma\sqrt{T} } { f\left(\vartheta
    _0\right)} \int_{v-\nu _\zT }^{v}\left[ e^{-a\left(\vartheta
    _0\right)\left(v-z\right)T }-e^{-r \left(\vartheta
    _0\right)\left(v-z\right)T }   \right]\;{\rm
    d}\bar W_T\left(z\right)\left(1+o\left(1\right)\right) .
\end{align*}

Let us denote
\begin{align*}
 Q_T\left(v\right)& =\sqrt{T}\int_{v-\nu 
    _\zT}^{v}e^{-r\left(\vartheta 
    _0\right)\left(v-u\right)T} u^{-1}  \eta _\zT\left(u\right)   {\rm d}\bar
  W_{T}\left(u\right),\\
U_T\left(v\right)& = \eta _\zT\left(v\right) \sqrt{T}\int_{v-\nu 
    _\zT}^{v}e^{-r\left(\vartheta 
    _0\right)\left(v-u\right)T}   {\rm d}\bar
  W_{T}\left(u\right),\\
g\left(\vartheta _0,zT\right)&=\frac{\dot M\left(\vartheta _0,
zT\right) }{\sigma\sqrt{{\rm I}\left(\vartheta
    _0\right)}},
\end{align*}
and calculate $\Ex_{\vartheta
  _0}\left[Q_T\left(v\right)-v^{-1} U_T\left(v\right)\right]^2$.
Recall that $\Ex_{\vartheta _0}g\left(\vartheta _0,
zT\right)^2=1 $ and 
\begin{align*}
\eta _\zT\left(v\right)=\frac{1 }{\sigma\sqrt{{\rm I}\left(\vartheta
    _0\right)}}\int_{\varepsilon _\zT}^{v}\dot M\left(\vartheta _0,
zT\right){\rm d}\bar W_T\left(z\right)=\int_{\varepsilon _\zT}^{v}g\left(\vartheta _0,
zT\right){\rm d}\bar W_T\left(z\right).
\end{align*}

We have 
\begin{align*}
\Ex_{\vartheta _0}Q_T\left(v\right)^2&=T\int_{v-\nu
  _\zT}^{v}e^{-2r\left(\vartheta _0\right)\left(v-u\right)T} u^{-2}
\Ex_{\vartheta _0} \eta _\zT\left(u\right)^2 {\rm d}u \\
 &=T\int_{v-\nu
  _\zT}^{v}e^{-2r\left(\vartheta _0\right)\left(v-u\right)T} u^{-1} {\rm
  d}u\left(1+O\left(\varepsilon _\zT\right)\right)\\
 &=\frac{1}{2r\left(\vartheta _0\right)
  v}+T\int_{v-\nu _\zT}^{v}e^{-2r\left(\vartheta _0\right)\left(v-u\right)T}
\frac{\left(v-u\right)}{vu} {\rm d}u+O\left(\varepsilon _\zT\right)\\
 &=\frac{1}{2r\left(\vartheta _0\right)  v}+O\left(\nu _\zT\right)+O\left(\varepsilon
_\zT\right) .
\end{align*}
Note that the integrals
\begin{align*}
I_{\varepsilon _\zT}^{v-\nu _\zT}=  \int_{\varepsilon _\zT}^{v-\nu _\zT} g(\vartheta _0,zT){\rm d}\bar
W_T\left(z\right)\quad {\rm and}\quad J_{v-\nu _\zT}^{v}=\sqrt{T}\int_{v-\nu _\zT}^{v}e^{-r\left(\vartheta
  _0\right)\left(v-u\right)T} {\rm d}\bar
W_{T}\left(u\right)
\end{align*}
are independent random variables. This allows us to write
\begin{align*}
\Ex_{\vartheta _0}U_T\left(v\right)^2&=\Ex_{\vartheta
  _0}\left(\int_{\varepsilon _\zT}^{v} g(\vartheta _0,zT){\rm d}\bar
W_T\left(z\right)\right) ^2\left(\sqrt{T}\int_{v-\nu _\zT}^{v}e^{-r\left(\vartheta
  _0\right)\left(v-u\right)T} {\rm d}\bar
W_{T}\left(u\right)\right)^2\\
 &=\Ex_{\vartheta  _0}\left(I_{\varepsilon _\zT}^{v-\nu _\zT}+I_{v-\nu
  _\zT}^{v}   \right)^2\left(J_{v-\nu _\zT}^{v}\right)^2 \\
 &=\Ex_{\vartheta  _0}\left(I_{\varepsilon _\zT}^{v-\nu _\zT}
\right)^2\Ex_{\vartheta  _0}\left(J_{v-\nu _\zT}^{v}\right)^2+2  \Ex_{\vartheta  _0}\left(I_{\varepsilon _\zT}^{v-\nu _\zT}I_{v-\nu
  _\zT}^{v}  
\right)\left(J_{v-\nu _\zT}^{v}\right)^2\\
 & \qquad \qquad +\Ex_{\vartheta  _0}\left(I_{v-\nu
  _\zT}^{v}  
\right)^2\left(J_{v-\nu _\zT}^{v}\right)^2.
\end{align*}
We have the following expressions  for the first term
\begin{align*}
\Ex_{\vartheta  _0}\left(I_{\varepsilon _\zT}^{v-\nu _\zT}
\right)^2\Ex_{\vartheta  _0}\left(J_{v-\nu _\zT}^{v}\right)^2=\left(v-\nu
_\zT-\varepsilon _\zT\right) \frac{1}{2r\left(\vartheta
  _0\right)}\left(1+o\left(1\right)\right) ,
\end{align*}
for the second term
\begin{align*}
\Ex_{\vartheta _0}\left(I_{\varepsilon _\zT}^{v-\nu _\zT}I_{v-\nu _\zT}^{v}
\right)\left(J_{v-\nu _\zT}^{v}\right)^2\leq \left[\Ex_{\vartheta
    _0}\left(I_{v-\nu 
  _\zT}^{v}\right)^2 \Ex_{\vartheta _0}\left(I_{\varepsilon _\zT}^{v-\nu
  _\zT}\right)^2\left(J_{v-\nu _\zT}^{v}\right)^4 \right]^{1/2}\leq C\nu
_\zT^{1/2} ,
\end{align*}
and for the third term
\begin{align*}
\Ex_{\vartheta _0}\left(I_{v-\nu _\zT}^{v} \right)^2\left(J_{v-\nu
  _\zT}^{v}\right)^2&\leq \left[\Ex_{\vartheta _0}\left(I_{v-\nu _\zT}^{v}
\right)^4\Ex_{\vartheta _0}\left(J_{v-\nu _\zT}^{v}\right)^4\right]^{1/2}\\
&\leq
C \left[ \Ex_{\vartheta _0}  \left(\int_{v-\nu _\zT}^{v} g(\vartheta
  _0,zT)^2{\rm d}z\right) ^2  \right]^{1/2}\leq C\nu _\zT. 
\end{align*}

Hence
\begin{align*}
\Ex_{\vartheta _0}\left[Q_T\left(v\right)-v^{-1} U_T\left(v\right)\right]^2=O\left(\varepsilon _\zT\right)+O(\nu _\zT^{1/2})\longrightarrow 0.
\end{align*}
Therefore we obtained the representation for the derivative
\begin{align}
\label{39a}
\dot m\left(\vartheta _0,vT\right)&=-\frac{\sigma \left[\dot a\left(\vartheta
    _0\right)+B\left(\vartheta _0\right)\dot f\left(\vartheta
    _0\right)\right]}{f\left(\vartheta _0\right)\sqrt{2a\left(\vartheta
    _0\right){\rm I}\left(\vartheta _0\right)}}\,\xi _T^{\left(1\right)}\left(v\right) +\frac{\sigma \dot
  r\left(\vartheta _0\right)}{f\left(\vartheta
  _0\right)\sqrt{2r\left(\vartheta _0\right){\rm I}\left(\vartheta _0\right)}} \,\xi
_T^{\left(2\right)}\left(v\right)+o\left(1\right)\nonumber\\
&=K_1\left(\vartheta _0\right)\,\xi _T^{\left(1\right)}\left(v\right) +K_2\left(\vartheta _0\right)\,\xi _T^{\left(2\right)}\left(v\right)+o\left(1\right).
\end{align}
Now the relations \eqref{69} and \eqref{70} follow from the representation
\begin{align}
\label{39b}
\sqrt{T}\left(m_{vT,T}^\star-m\left(\vartheta _0,t\right)\right)=\frac{\eta
  _T\left(v\right)}{v}\sum_{i=1}^{2} K_i\left(\vartheta _0\right)\xi
_\zT^{\left(i\right)}\left(v\right)+o\left(1\right). 
\end{align}
and the  properties of $\eta _T\left(v\right),$ $\xi
_\zT^{\left(i\right)}\left(v\right),i=1,2$

\end{proof}

\begin{remark}
\label{R2}
{\rm Suppose that the unknown parameter is $\vartheta =b$ and write the
  adaptive filter for the system
\begin{align*}
{\rm d}X_t&=f\,Y_t\,{\rm d}t+\sigma\, {\rm d}W_t,\qquad
X_0,\qquad 0\leq t\leq T,\\
{\rm d}Y_t&=-a\,Y_t\,{\rm d}t+\vartheta \, {\rm
  d}V_t,\qquad  Y_0,\qquad t\geq 0,
\end{align*}
in recurrent form. 

Recall that $\tau _\zT=T^\delta ,\delta \in(\frac{1}{2},1)$, the preliminary estimator and One-step MLE process are 
\begin{align*}
\vartheta^* _{\tau _\zT}&=\left(\frac{a^3}{f^2\left[e^{-a}-1+a\right]\tau
  _\zT} \sum_{k=1}^{\tau
  _\zT}\left(\left[X_{k}-X_{k-1}\right]^2-\sigma^2\right)
\right)^{1/2},\\ 
{\rm d}\vartheta _{t,T}^\star&=\frac{\vartheta^* _{\tau
    _\zT}-\vartheta _{t,T}^\star}{t-\tau _\zT}\,{\rm d}t+\frac{2\sigma
  ^3\,r(\vartheta^* _{\tau _\zT} )^3\,\dot M(\vartheta^* _{\tau
    _\zT},t)}{f^4\,\vartheta^{*2} _{\tau _\zT}\,\left(t-\tau _\zT \right)}\,
\left[{\rm d}X_t-M(\vartheta^* _{\tau _\zT},t){\rm d}t\right],\quad \tau
_\zT<t\leq T.
\end{align*}
The equations for $ M(\vartheta^* _{\tau _\zT},t),\tau _\zT<t\leq T $ and
$\dot M(\vartheta^* _{\tau _\zT},t),\tau _\zT<t\leq T$ are
\begin{align*}
 {\rm d}M(\vartheta^* _{\tau _\zT},t)&=-r(\vartheta^* _{\tau
   _\zT})M(\vartheta^* _{\tau _\zT},t){\rm d}t+\Gamma (\vartheta^*
 _{\tau _\zT}){\rm d}X_t,\qquad M(\vartheta^* _{\tau _\zT},0),\\
 {\rm d}\dot M(\vartheta^* _{\tau _\zT},t)&=-r(\vartheta^* _{\tau
   _\zT})\dot M(\vartheta^* _{\tau _\zT},t){\rm d}t+\dot r(\vartheta^* _{\tau
   _\zT}) M(\vartheta^* _{\tau _\zT},t){\rm d}t+
\dot \Gamma (\vartheta^*
 _{\tau _\zT}){\rm d}X_t,\qquad \dot M(\vartheta^* _{\tau _\zT},0),
\end{align*}
where the initial values $M(\vartheta^* _{\tau _\zT},0) $ and $\dot
M(\vartheta^* _{\tau _\zT},0) $ are calculated by formulas \eqref{53} and
\eqref{54} respectively.

}
\end{remark}

\begin{remark}
\label{R3}
{\rm  The approximation of $ m\left(\vartheta ,t\right)$ in the case of
  two-dimensional parameter, say, in the case {\bf AB}:  $\vartheta
  =\left(a,b\right)^\top=\left(\theta _1,\theta _2\right)$ with Fisher
  information matrix  can be written too, but the expressions
  will be more cumbersome and the proof quite close to the given above proof in
  the   one-dimensional case.

}
\end{remark}
\section{Asymptotic efficiency}

We have the same partially observed system
\begin{align*}
{\rm d}X_t&=f\left(\vartheta \right)\,Y_t\,{\rm d}t+\sigma\, {\rm d}W_t,\qquad
X_0,\qquad 0\leq t\leq T,\\
{\rm d}Y_t&=-a\left(\vartheta \right)\,Y_t\,{\rm d}t+b\left(\vartheta \right) \, {\rm
  d}V_t,\qquad  Y_0,\qquad t\geq 0,
\end{align*}
where the functions $f\left(\cdot \right),a\left(\cdot \right),b\left(\cdot
\right)$ satisfy the conditions of  Theorem \ref{T1}. The adaptive filter
is  given by \eqref{68} and we would like to know if the error of
approximation $\Ex_\vartheta \left|m_{t,T}^\star-m\left(\vartheta,t
\right)\right|^2 $ is asymptotically minimal?  As before we propose a lower
minimax bound on the risks of all estimators $\bar m_t$ supposing that their
calculation is based on the observations up to time $t$, i.e., these
estimators depend on $X^t=\left(X_s,0\leq s\leq t\right)$.

By Theorem \ref{T1} the MLE and BE are consistent, asymptotically normal and
the polynomial moments converge uniformly on compacts $\KK\subset\Theta
$. Moreover, as it follows from Theorem 3.2.2 in \cite{IH81} the BE admits the
representation similar to \eqref{58a}
\begin{align}
\label{6-41}
\sqrt{t}\left(\tilde\vartheta _{t}-\vartheta _0\right)=\frac{1}{\sigma{\rm
    I}\left(\vartheta _0\right) \sqrt{t}}\int_{0}^{t}\dot M\left(\vartheta
_0,s\right) \,{\rm d}\bar W_s +o\left(1\right).
\end{align}
Denote 
\begin{align*}
S^\star\left(\vartheta _0\right)^2=K_1\left(\vartheta _0\right)^2+K_2\left(\vartheta
_0\right)^2+R_{1,2}\left(\vartheta _0\right). 
\end{align*}

\begin{theorem}
\label{T4}
Let the conditions of Theorem \ref{T1} be fulfilled. Then 
we have the following lower minimax bound: for any estimator $\bar m_{t,T}$ of
$m\left(\vartheta,t \right)$ (below $t=vT$)
\begin{align*}
\lim_{\nu \rightarrow 0}\Liminf_{T\rightarrow \infty }\sup_{\left|\vartheta
  -\vartheta _0\right|\leq \nu }t\, \Ex_{\vartheta }\left|\bar
m_{t,T}-m\left(\vartheta,t \right) \right|^2 \geq S^\star\left(\vartheta _0\right)^2.
\end{align*}

\end{theorem}
\begin{proof} The given below proof is based on the proof of Theorem 1.9.1 in \cite{IH81}.
We have the elementary estimate
\begin{align*}
\sup_{\left|\vartheta -\vartheta _0\right|\leq \nu } \Ex_{\vartheta
}\left|\bar m_{t,T}-m\left(\vartheta,t \right) \right|^2 \geq \int_{\vartheta
  _0-\nu }^{\vartheta _0+\nu }\Ex_{\vartheta }\left|\bar
m_{t,T}-m\left(\vartheta,t \right) \right|^2p_\nu \left(\vartheta \right){\rm
  d}\vartheta.
\end{align*}
Here the function $p_\nu \left(\vartheta \right),\vartheta _0-\nu<\vartheta
<\vartheta _0+\nu $ is a positive continuous density. If we denote
$\tilde m_{t}$ Bayesian estimator of $m\left(\vartheta,t \right)  $, which
corresponds to this density $p_\nu \left(\cdot \right)$, then 
\begin{align*}
\tilde m_{t}=\int_{\vartheta _0-\nu }^{\vartheta _0+\nu }m\left(\theta,t
\right) p_\nu \left(\theta |X^t\right){\rm d}\theta  ,\qquad p_\nu
\left(\theta |X^t\right)=\frac{p_\nu \left(\theta
  \right)L\left(\theta ,X^t\right) }{\int_{\vartheta _0-\nu }^{\vartheta
    _0+\nu }p_\nu \left(\theta 
  \right)L\left(\theta ,X^t\right){\rm d}\theta   }
\end{align*}
and
\begin{align*}
\int_{\vartheta
  _0-\nu }^{\vartheta _0+\nu }\Ex_{\vartheta }\left|\bar
m_{t,T}-m\left(\vartheta,t \right) \right|^2p_\nu \left(\vartheta \right){\rm
  d}\vartheta\geq \int_{\vartheta
  _0-\nu }^{\vartheta _0+\nu }\Ex_{\vartheta }\left|\tilde
m_{t}-m\left(\vartheta,t \right) \right|^2p_\nu \left(\vartheta \right){\rm
  d}\vartheta.
\end{align*}
The asymptotic behavior of BE $\tilde m_{t} $ can be described as follows (below $\theta
_u=\vartheta+\varphi _t u, \varphi _t=t^{-1/2}, \UU_\nu
=\left(\sqrt{t}\left(\vartheta _0-\nu -\vartheta
\right),\sqrt{t}\left(\vartheta _0+\nu -\vartheta \right) \right)$    )  
\begin{align*}
\tilde m_{t}&=\frac{\int_{\vartheta _0-\nu }^{\vartheta _0+\nu }m\left(
  \theta,t \right){p_\nu \left(\theta \right)L\left(\theta ,X^t\right) }{\rm
    d}\theta } {\int_{\vartheta _0-\nu }^{\vartheta _0+\nu }p_\nu \left(\theta
  \right)L\left(\theta ,X^t\right){\rm d}\theta }=\frac{\int_{\UU_\nu }^{
  }m\left(\theta_u,t \right){p_\nu \left(\theta_u
    \right)L\left(\theta_u,X^t\right) }{\rm d}u } {\int_{\UU_\nu }^{}p_\nu
  \left(\theta_u \right)L\left(\theta_u ,X^t\right){\rm d}u }\\
 &=m\left(\vartheta,t \right)+\varphi _t\dot m\left( \vartheta,t \right)
\frac{\int_{\UU_\nu }^{ }u{p_\nu \left(\theta_u
    \right)\frac{L\left(\theta_u,X^t\right)}{L\left(\vartheta ,X^t\right)}
  }{\rm d}u } {\int_{\UU_\nu }^{}p_\nu \left(\theta_u
  \right)\frac{L\left(\theta_u,X^t\right)}{L\left(\vartheta ,X^t\right)}{\rm
    d}u }\left(1+o\left(1\right)\right)\\
 &=m\left(\vartheta,t \right)+\varphi _t\dot m\left( \vartheta,t \right)
\frac{\int_{\UU_\nu } u  p_\nu \left(\vartheta \right)Z_t\left(u\right){\rm
      d}u } {\int_{\UU_\nu }^{}p_\nu \left(\vartheta
    \right)Z_t\left(u\right){\rm d}u }\left(1+o\left(1\right)\right).
\end{align*}
Hence
\begin{align}
\sqrt{t}\left(\tilde m_{t}- m\left(\vartheta,t \right)
\right)&=\dot m\left( \vartheta,t \right)\frac{\int_{\UU_\nu } u
  Z_t\left(u\right){\rm d}u } {\int_{\UU_\nu 
  }^{} Z_t\left(u\right){\rm d}u }\left(1+o\left(1\right)\right)\\
&=\dot m\left( \vartheta,t \right)\frac{\Delta _t\left(\vartheta
  ,X^t\right)}{{\rm I}\left(\vartheta \right)} \left(1+o\left(1\right)\right)  .
\end{align}
Here
\begin{align*}
\Delta _t\left(\vartheta ,X^t\right)=\frac{1}{\sigma \sqrt{t}}\int_{0}^{t}\dot
M\left(\vartheta ,s\right) \,{\rm d}\bar W_s\Longrightarrow \Delta \sim{\cal
  N}\left(0,{\rm I}\left(\vartheta \right) 
 \right)
\end{align*}
Recall that
\begin{align*}
Z_t\left(u\right)\Longrightarrow Z\left(u\right)=\exp\left(u\Delta
\left(\vartheta \right)-\ds\frac{u^2}{2}{\rm I}\left(\vartheta \right) \right),\qquad \frac{\int_{{\cal R} } u Z\left(u\right){\rm d}u } {\int_{{\cal R} }^{}
  Z\left(u\right){\rm d}u }=\frac{\Delta
\left(\vartheta \right)}{{\rm I}\left(\vartheta \right) }.
\end{align*}
  Moreover the uniform on compacts
$\KK\subset \left(\vartheta _0-\nu ,\vartheta _0+\nu\right)$ convergence of
moments
\begin{align}
\label{78a}
t\Ex_\vartheta \left(\tilde m_{t}- m\left(\vartheta,t
\right) \right)^2 \rightarrow \lim_{t\rightarrow \infty }t \Ex_\vartheta \left[\dot
m\left( \vartheta,t \right)^2(\tilde\vartheta _t-\vartheta )^2\right]=S^\star\left(\vartheta \right)^2
\end{align}
holds too. 
The detailed proof of written above relations can be found in the proofs of
Theorems 3.2.1 and 3.2.2 in \cite{IH81} 

Therefore
\begin{align*}
t\int_{\vartheta _0-\nu }^{\vartheta _0+\nu }\Ex_{\vartheta }\left|\tilde
m_{t}-m\left(\vartheta,t \right) \right|^2p_\nu \left(\vartheta \right){\rm
  d}\vartheta \longrightarrow \int_{\vartheta _0-\nu }^{\vartheta _0+\nu
}S^\star\left(\vartheta \right)^2p_\nu \left(\vartheta \right){\rm d}\vartheta 
\end{align*}
and as $\nu \rightarrow 0$ 
\begin{align*}
\int_{\vartheta _0-\nu }^{\vartheta _0+\nu
}S^\star\left(\vartheta \right)^2p_\nu \left(\vartheta \right){\rm d}\vartheta
\longrightarrow S^\star\left(\vartheta_0 \right)^2 .
\end{align*}

\end{proof}

  We call the estimator $ m_{t,T}^\circ, \tau _\zT<t\leq T $   {\it
  asymptotically efficient} if for all $\vartheta _0\in\Theta $, $t=vT$ and  any  $v\in
  \left[\varepsilon _0,1\right]$
\begin{align*}
\lim_{\nu \rightarrow 0}\lim_{T\rightarrow \infty }\sup_{\left|\vartheta
  -\vartheta _0\right|\leq \nu }t \Ex_{\vartheta }\left|
m_{t,T}^\circ -m\left(\vartheta,t \right) \right|^2 = S^\star\left(\vartheta _0\right)^2.
\end{align*}

Here $\varepsilon _0\in (0,1)$.

\begin{theorem}
\label{T7}
Let the conditions of Theorem \ref{T1} be fulfilled. Then the estimator
$m^\star_t, $ $\tau _\zT<t\leq T$ is asymptotically efficient.
\end{theorem}
\begin{proof}
Remind that in this regular statistical problem the BE and MLE have   the same
asymptotic properties (see Theorem \ref{T1}) and the One-step MLE-process
 is in this sense equivalent to the BE too. Therefore if we replace
$\tilde\vartheta _t$ by $\vartheta_{t,T}^\star $ in \eqref{78a} and compare
 the representations  \eqref{58a} and \eqref{6-41}, then we conclude that 
 the  limits   have to be the same.  We have
\begin{align*}
m^\star_t-m\left(\vartheta ,t\right)=\dot
m\left(\vartheta ,t\right)\, \left(\vartheta_{t,T}^\star-\vartheta
\right)\left(1+o\left(1\right)\right)= \dot
m\left(\vartheta ,t\right)\, (\tilde\vartheta_{t}-\vartheta
)\left(1+o\left(1\right)\right).
\end{align*}

\end{proof}



\begin{thebibliography}{99}

\bibitem{A83} Arato, M. (1983) {\it Linear Stochastic Systems with Constant
  Coefficients A Statistical Approach.} Lecture Notes in Control and
  Inform. Sci., 45, New York: Springer-Verlag.

\bibitem{BaB84} Bagchi, A.,  Borkar,V. (1984) Parameter  identification  in  infinite
 dimensional linear systems. {\it Stochastics}, { 12}, 201--213.


\bibitem{BC93}  Bell, B.M. and Cathey, F. (1993)   The iterated Kalman filter update as a
Gauss-Newton method. {\it  IEEE Trans. Autom. Control},  38,  2, 294-297.

\bibitem {BRR98} Bickel, P.J., Ritov, Y. and Ryd\'en, T. (1998) Asymptotic
  normality of the maximum likelihood estimator for general hidden Markov
  models. {\it Ann. Statist.}, 26, 4, 1614-1635.

\bibitem {BRR85} Brown, S.D. and Rutan, S.C. (1985) Adaptive Kalman
  filtering. {\it Journal of  Research of the national bureau of standards.}
  90, 6, 403-407.

\bibitem {CMR05} Capp\'e, O., Moulines, E. and Ryd\'en, T. (2005) {\it Inference
  in Hidden Markov Models}. Springer, N.Y.



\bibitem {DZ86} Dembo, A, Zeitouni, O. (1984) Parameter estimation of
  partially observed continuous time stochastic processes vie EM
  algorithm. {\it Stochastic  Process. Appl.}, 23, 91-113.

\bibitem {Hay14} Haykin, S. (2014) {\it Adaptive Filter Theory.} Fifth Ed.,
  Pearson Education, Boston.  


\bibitem{IH81} Ibragimov, I.A. and Khasminskii R. Z. (1981)
 {\it Statistical Estimation --- Asymptotic Theory.} Springer, N.Y.

\bibitem{Jaz70} Jazwinski, A.H. (1970) {\it Stochastic Processes and Filrering
  Theory.} Academic Press, N.Y. 


\bibitem{KB61} Kalman, R.E., and Bucy, R.S. (1961) New results in linear
  filtering and prediction theory. {\it Trans. ASME}, 83D, 95-100.

\bibitem{KS91} Kallianpur, G., Selukar,  R.S.  (1991)   Parameter estimation in
linear filtering. {\it  J. Multivar Anal.}, {\bf 39}, 284--304.


\bibitem{KhK18} Khasminskii, R.Z. and Kutoyants, Yu.A.  (2018) On parameter
 {\bf Z} estimation of hidden telegraph process.  {\it Bernoulli}, 24, 3,
  2064-2090.

 \bibitem{Kut84} Kutoyants, Yu.A. (1984) {\it Parameter Estimation for
  Stochastic Processes.} Heldermann, Berlin. 

\bibitem{Kut94} Kutoyants, Yu.A. (1994) {\it Identification of Dynamical
  Systems with Small Noise.} Kluwer Academic Publisher, Dordrecht.

\bibitem{Kut04} Kutoyants, Yu.A.\; (2004)\; {\it Statistical Inference for
  Ergodic Diffusion Processes.} Springer, London. 


\bibitem{Kut19b} Kutoyants, Yu.A.  (2019)  On parameter estimation of the hidden
  ergodic Ornstein--Uhlenbeck process. {\it Electronic Journal of
    Statistics}, 13, 4508-4526. 

\bibitem{Kut19a} Kutoyants, Yu.A.  (2019)  On parameter estimation of the hidden
  Ornstein--Uhlenbeck process. {\it J. Multivar. Analysis}, 169, 1, 248-269.


\bibitem{Kut22} Kutoyants, Yu. A.  (2022) Volatility estimation of hidden
  Markov process and adaptive filtration", {\it arXiv:2010.07603}, submitted

\bibitem{Kut23b} Kutoyants, Yu. A.  (2023) Hidden AR process and adaptive
  Kalman filter. Submitted. 


\bibitem{KM16} Kutoyants, Yu.A. and Motrunich, A.  (2016) On multi-step
  MLE-process for Markov sequences. {\it Metrika}, 79, 705-724.

\bibitem{KZ21} Kutoyants, Yu.A. and Zhou, L. (2021) On parameter estimation
  of the hidden Gaussian process in perturbed SDE, {\it Electr. J. of
    Stat.}, 15, 211-234.

\bibitem{LS01} Liptser, R.S. and  Shiryayev, A.N.  (2001) {\it Statistics of
  Random Processes, I. General Theory.} 2nd Ed., Springer, N.Y.

\bibitem{Liung79} Ljung, L. (1979) Asymptotic behavior of the extended Kalman
  filter as a parameter estimator for linear systems. {\it IEEE Trans. Autom. Control} , 24(1), 36-50.


\bibitem{Meh70} Mehra, R.K.  (1970) On the identification of variances and
  adaptive Kalman filtering, {\it IEEE Trans. Autom. Control AC-IS,} 175-184,

\bibitem{S08}   Sayed, A.H. (2008) {\it Adaptive Filters.} John Wiley \& Sons,
  New Jersey.
\end{thebibliography}
\end{document}